\documentclass[11pt,a4paper]{article}
\usepackage{graphicx}
\usepackage{amsfonts}
\usepackage{amssymb}
\usepackage{amsthm}
\usepackage[centertags]{amsmath}
\usepackage{newlfont}

\usepackage{color}
\usepackage{graphicx,color}
\usepackage{amsmath, amssymb, graphics}

\def\r{\mathbb R}
\def\l{\mathbb L}
 \def\e{\mathbb E}
 
\DeclareMathOperator{\arccosh}{arcosh}
\DeclareMathOperator{\arcsinh}{arcsinh}

\newtheorem{theorem}{Theorem}[section]

\newtheorem{proposition}[theorem]{Proposition}
\newtheorem{definition}[theorem]{Definition}
\newtheorem{lemma}[theorem]{Lemma}

\title{Classification of zero mean curvature surfaces of separable type in Lorentz-Minkowski space}

\author{Seher Kaya\\
Department of Mathematics\\
 Ankara University\\
Ankara, Turkey\\
\texttt{seherkaya@ankara.edu.tr}
\and
Rafael L\'opez\footnote{Partially supported by the grant no. MTM2017-89677-P, MINECO/AEI/FEDER, UE}\\
 Departamento de Geometr\'{\i}a y Topolog\'{\i}a\\ Instituto de Matem\'aticas (IEMath-GR)\\
 Universidad de Granada\\
 18071 Granada, Spain\\
\texttt{rcamino@ugr.es}}

\date{}
\begin{document}
\maketitle

\begin{abstract} Consider  the Lorentz-Minkowski $3$-space $\l^3$ with the metric $dx^2+dy^2-dz^2$ in canonical coordinates $(x,y,z)$. A   surface in $\l^3$ is said to be separable if satisfies an equation of the form $f(x)+g(y)+h(z)=0$ for  some smooth functions $f$, $g$ and $h$ defined in open intervals of the real line. In this article we classify  all zero mean curvature surfaces of separable type, providing a method of construction of examples. 
\end{abstract}

\section{Introduction}

A zero mean curvature (ZMC, for short) surface in the Lorentz-Minkowski $3$-space $\l^3$ is a non-degenerate surface whose mean curvature is zero   at every point of the surface. There are two types of non-degenerate surfaces, namely,  spacelike and timelike surfaces if the induced metric is Riemannian and Lorentzian respectively. A spacelike   ZMC surface is   called a maximal surface   because it locally maximizes the area functional. A timelike ZMC surface is also called a   timelike minimal surface.  In general, we will allow ZMC surfaces of mixed type, that is,  it may exist regions on the surface of spacelike type and  timelike type. On the other hand, the zero mean curvature equation in non-parametric form   is a PDE which is of elliptic type (resp. hyperbolic type) in case the surface is spacelike (resp. timelike).

  One of the problems in the theory of ZMC surfaces  is the construction of     examples. Among the different techniques commonly  proposed in the literature,  we point out the use of complex analysis by means of the Weierstrass representation \cite{ko,we} and the Bj\"{o}rling formula \cite{ac,cd} or the use of integrable systems in the method of  loop groups \cite{dt,it}. It is also of special interest    to provide examples of   surfaces with particular   geometric properties.  A clear example of this idea is the family of ZMC surfaces that are invariant by a uniparametric group of rotations of $\l^3$.  These surfaces, the so-called rotational ZMC surfaces,   are classified: see for example \cite{ko,lo1,va1}. More recently, there is a great activity in the search of explicit examples of ZMC surfaces of   mixed type  where the transition between the spacelike and timelike regions occurs along degenerate curves: here we refer \cite{as,fu0,fu1,fu2,fu3,fu4,kim,st} without to be a complete list.   

In this paper we present a new method for constructing ZMC surfaces   in Lorentz-Minkowski space $\l^3$ by the technique of separable variables.  The Lorentz-Minkowski $\l^3$ is the vector space $\r^3$ with canonical coordinates $(x,y,z)$ and endowed  with metric $\langle,\rangle$ of signature $(+,+,-)$. Any surface of $\l^3$ is locally the zero level set $F(x,y,z)=0$ of a smooth function $F$ defined in an open set $O\subset\r^3$, being $0$ a regular value of $F$. Our strategy in the construction method consists in to assume  that the function     $F$ is     a separable function of the three variables $x$, $y$ and $z$. This transforms the   zero mean curvature equation    into an ODE which may   be more manageable.

\begin{definition} A   surface $S$ in $\l^3$ is said to be separable if  can be expressed as  
\begin{equation}\label{s1}
S=\{(x,y,z)\in\l^3:f(x)+g(y)+h(z)=0\}.
\end{equation}
 Here $f$, $g$ and $h$ are   smooth functions defined in some intervals $I_1$, $I_2$ and $I_3$ of $\r$, respectively. Also, $f'(x)^2+g'(y)^2+h'(z)^2\not=0$ for every $x\in I_1$, $y\in I_2$ and $z\in I_3$. 
 \end{definition}

In this paper, we pose the following

\begin{quote}{\bf Problem.} {\it Classify all    separable surfaces in $\l^3$ with   zero mean curvature.}\end{quote}

A first example of separable surface occurs if  one of the functions in (\ref{s1}) is   linear  in its variable. Without loss of generality, we suppose that $h$ is the linear function $h(z)=az+b$, with $a,b\in \r$. If $a=0$, then the implicit equation of the surface is $f(x)+g(y)+b=0$, which says that the surface is a cylindrical surface whose base curve is a planar curve contained in the $xy$-plane. If now we assume that the surface has zero mean curvature, then this planar curve is a straight-line and the surface  is  a plane. If $a\not=0$, then the implicit equation of the  surface is  $z= -f(x)/a-g(y)/a-b/a$. In the literature, a surface   expressed as  $z=\phi(x)+\psi(y)$ is called a translation surface. The  family of  ZMC surfaces  of translation type was classified in   \cite{ko,li,va1}. These surfaces  are the (spacelike and timelike) analogous  ones in the Lorentzian setting of the classical minimal surfaces in the Euclidean space discovered by Scherk in 1835 (\cite{sc}; see also \cite{ni}), together a flat B-scroll over a null curve which appears as an exceptional case when the surface is a timelike minimal surface \cite{va1}. In the rest of the paper we will assume that the surface is neither cylindrical nor translation type,  or equivalently, none of the functions $f$, $g$ and $h$ are linear functions. 
 
 Other examples of separable surfaces are the rotational surfaces when the rotation axis is one of the coordinate axes. Indeed, if the rotational axis is the $z$-line, the implicit equation of the surface   is  $G(z)=x^2+y^2$.  In case that the rotational axis is the $x$-line (resp. the $y$-line), then the rotational surface is $G(x)=y^2-z^2$ (resp. $G(y)=x^2-z^2$).  Rotational surfaces with zero mean curvature will appear as a particular case in our method of construction of separable ZMC surfaces: see the beginning of Section \ref{sec3}.
 
In the Euclidean space,  the minimal surfaces of separable type were initially studied by    Weingarten \cite{we} in 1887 with the purpose  to generalize the  translation surfaces $z=\phi(x)+\psi(y)$ obtained by Scherk \cite{sc}. Later, in 1956-7 Fr\'echet  gave a deep study of these surfaces obtaining explicit examples \cite{fr1,fr2}.  Our paper is inspired of the calculations of Fr\'echet (see also \cite[Sect. 5.2]{ni}). It deserves to point that Sergienko and Tkachev  \cite{st} gave an approach to the study of  separable maximal surfaces in $\l^3$ because  they  were interested in the construction of   doubly periodic maximal surfaces which satisfy an implicit equation of type $\zeta(z)=\phi(x)\psi(y)$, which is equivalent to (\ref{s1}).  

The goal of our paper is to investigate  the   problem in all its   generality, assuming mixed type causal character and   providing an exhaustive method to construct all separable ZMC surfaces. This allows to present a plethora of new examples of ZMC surfaces, including also other known surfaces    in the literature.

In Section \ref{sec2}, we compute the zero mean curvature equation  of a separable surface obtaining for each one of the functions $f$, $g$ and $h$ a differential equation of fourth order. After successive integrations, we finally    obtain  three ODEs of first order on $f$, $g$ and $h$ respectively depending on  a real constant $K$ and a set of nine constants $a_i$, $b_i$, $c_i$, $1\leq i\leq 3$. These constants are not arbitrary because they are linked by  six nonlinear equations. A first step in our strategy is to search the constants $a_i$, $b_i$ and $c_i$ satisfying the above six equations and once obtained these constants, to solve the ODEs.  In general, this system can not be solved by quadratures and the solutions will be expressed in terms of elliptic functions. We separate this discussion in the sections \ref{sec3}, \ref{sec4} and \ref{sec5} depending on the sign of the   constant $K$. We will obtain  explicit examples of separable ZMC surfaces and we will study the causal character of the surface and, if possible, its extension to regions of lightlike points  in each case.

\section{The method of construction  of separable ZMC surfaces}\label{sec2}

The Lorentz-Minkowski space $\l^3$ is   the vector space $\r^3$ with canonical coordinates  $(x,y,z)$ and endowed with the Lorentzian metric $\langle,\rangle=dx^2+dy^2-dz^2$. A vector $\vec{v}\in\r^3$ is said to be spacelike, timelike or lightlike if the inner product $\langle \vec{v},\vec{v}\rangle$ is positive, negative or zero, respectively. The norm of $\vec{v}$   is $|\vec{v}|=\sqrt{\langle \vec{v},\vec{v}\rangle}$ if $\vec{v}$ is spacelike and $|\vec{v}|=\sqrt{-\langle \vec{v},\vec{v}\rangle}$ if $\vec{v}$ is timelike.  More generally,  a surface (or a curve) $S$  of $\l^3$  is called  spacelike, timelike or lightlike if the induced metric on $S$ is Riemannian, Lorentzian or degenerated, respectively. This property   is called the causal character of $S$.    We refer  the reader to \cite{lo2,ws} for some basics of $\l^3$. In order to have no confusion,    we denote by $\e^3$ for the Euclidean $3$-space, that is, $\r^3$  endowed with the Euclidean metric $dx^2+dy^2+dz^2$.

Let $S$ be a surface in $\l^3$  whose induced metric $\langle,\rangle$ is  non-degenerated.    Recall that it is equivalent to say that $S$ is spacelike (resp. timelike) if there is a unit normal timelike (resp. spacelike)  vector field defined on $S$.  Here we allow that the surface is of mixed type, so it may exist regions on the surface of spacelike and timelike type and, eventually, we will study if the surface can be extended to regions of $\l^3$ of lightlike type. In both types of surfaces, the mean curvature $H$   is defined as the trace of the second fundamental form. If $H=0$ everywhere, we say that $S$ has zero mean curvature and we abbreviate by saying a ZMC surface. 

We know that any surface $S$   is locally given by an implicit equation $F(x,y,z)=0$, where $0$ is a regular value of $F$.   The mean curvature $H$ is  calculated by means of  the gradient   and the Hessian matrix of $F$ where the computations are similar as in   $\e^3$.   The Lorentzian gradient   of $F$ is  

$$\nabla^{L}F=\left( F_{x}, F_{y},- F_{z}\right)$$
 where, as it is usual, the subscripts indicate the   partial derivatives with respect to the  corresponding variable.  The surface is spacelike (resp. timelike) if $\langle \nabla^{L}F,\nabla^{L}F\rangle<0$ (resp. $\langle \nabla^{L}F,\nabla^{L}F\rangle>0$) and consequently  $N=\nabla^LF/|\nabla^LF|$ defines  a unit normal vector field on $S$. The mean curvature $H$ is  the Lorentzian divergence $\mbox{div}^L$ of $N$,   
$$ \mbox{div}^L\left(\frac{\nabla^{L}F}{|\nabla^{L}F|}\right)=\bigg(\frac{F_{x}}{|\nabla^{L}F|}\bigg)_{x}+\left(\frac{F_{y}}{|\nabla^{L}F|}\right)_{y}-\left(\frac{F_{z}}{|\nabla^{L}F|}\right)_{z}=H.$$
Therefore the ZMC equation $H=0$ is equivalent to
$$ \frac{\Delta^{L}F}{|\nabla^{L}F|}+\frac{\epsilon}{|\nabla^{L}F|^{3}} (\nabla^LF)^t\cdot \mbox{Hess} F\cdot \nabla^LF=0,$$
where    $\epsilon=1$ (resp. $\epsilon=-1$)  if $S$ is  spacelike   (resp. timelike),  $\Delta^LF=F_{xx}+F_{yy}-F_{zz}$ and
$$\mbox{Hess}F=
\left(
\begin{array}{ccc}
 F_{xx}&F_{xy}&F_{xz}  \\
 F_{yx}&F_{yy}&F_{yz}\\
 F_{zx}&F_{zy}&F_{zz}
\end{array}
\right).$$

\begin{proposition}
If   $S= F^{-1}(\{0\})$ is a non-degenerate surface in  $\l^3$, then  $H=0$    if and only if
\begin{equation}\label{elm1}
-\langle \nabla^{L}F,\nabla^LF\rangle\Delta^{L}F+ (\nabla^LF)^t\cdot \mbox{Hess} F\cdot \nabla^LF=0.
\end{equation}
\end{proposition}

Now suppose  that the function $F=F(x,y,z)$ is of separable variables   $F(x,y,z)=f(x)+g(y)+h(z)$.   The causal character of $S$ is determined by the sign of $\langle\nabla^L F,\nabla^L F\rangle=f'^2+g'^2-h'^2$, where the symbol $'$ indicates the  derivative with respect to the corresponding variable. Thus if $S$ is   spacelike (resp. timelike)   then   $f'^2+g'^2-h'^2<0$ (resp. $f'^2+g'^2-h'^2>0$). The  equation $H=0$ in (\ref{elm1}) is now
\begin{equation}\label{eq1}
f''(g'^2-h'^2)+g''(f'^2-h'^2)-h''(f'^2+g'^2)=0.
\end{equation}

Let us introduce the   notation  
$$u=f(x), \quad v=g(y),\quad w=h(z)$$ 
and 
$$X(u)=f'^2,\quad Y(v)=g'^2,\quad Z(w)=h'^2.$$
In the new variables, the implicit equation of the surface (\ref{s1})  is now $u+v+w=0$. In terms of the   functions $X$, $Y$ and $Z$, the causal character of the surface is determined by the sign of $X(u)+Y(v)-Z(w)$, with $u+v+w=0$, which is spacelike (resp. timelike, lightlike) if the sign is negative (resp. positive, zero). 

The ZMC equation  (\ref{eq1}) can be   expressed as
\begin{equation}\label{eq2}
A:=(Y-Z)X'+(X-Z)Y'-(X+Y)Z'=0,
\end{equation}
for all values $u$, $v$ and $w$ with the condition $u+v+w=0$. 

Since we are assuming that the separable surface is not       cylindrical or  translation type, then none of the functions $f$, $g$ and $h$ are linear, in particular, none of the three functions $X'$, $Y'$ or $Z'$ can vanish identically in some open interval of the corresponding domain. 

We need the following auxiliary result.

\begin{lemma} \label{le}
Let $Q=Q(u,v,w)$ be a smooth function defined in a domain $\Omega\subset\r^3$. If $Q(u,v,w)=0$ for any triple of the section $\Omega\cap\Pi$, where $\Pi$ is the plane of equation $u+v+w=0$, then on the section we have
$$Q_u=Q_v=Q_w.$$
\end{lemma}
\begin{proof}
If   we write $w=-u-v$, then $Q(u,v,-u-v)=0$. Differentiating with respect to $u$, we deduce $Q_u-Q_w=0$.   Changing the roles of $u$, $v$ and $w$, we conclude the result. 
\end{proof}

Using this lemma,   the identities  $A_u-A_v=0$, $A_v-A_w=0$ and $A_u-A_w=0$, are respectively  
\begin{equation}\label{eq3}
\begin{split}
B_1:=&(Y-Z)X''-(X-Z)Y''-(X'-Y')Z'=0,\\
B_2:=&(Y'+Z')X'+(X-Z)Y''+Z''(X+Y)=0,\\
B_3:=&(Y-Z)X''+(X'+Z')Y'+(X+Y)Z''=0.
\end{split}
\end{equation}
From \eqref{eq2} and  the first two equations in \eqref{eq3}, we have a system of linear equations on   $Y-Z$, $X-Z$ and $X+Y$. The determinant of the coefficients of this system  is $-M$, where 
$$M=X''Y''Z'+X'Y''Z''+X''Y'Z''.$$
We find directly that  
$$M (Y-Z)X'=-X'Y'Z'(Y''(X'+Z')+Z''(Y'-X')),$$
$$M (X-Z)Y'=-X'Y'Z'(X''(Y'+Z')-Z''(Y'-X')),$$
$$M(X+Y)Z'=-X'Y'Z'(X''(Y'+Z')+Y''(X'+Z')). $$

Applying Lemma \ref{le} again to the functions $B_1$, $B_2$ and $B_3$ in (\ref{eq3}), we have $(B_1)_v-(B_1)_w=0$, $(B_2)_u-(B_2)_w=0$ and $(B_3)_u-(B_3)_v=0$. These three equations write, respectively, 
$$X''(Y'+Z')-Y'''(X-Z)-Z''(Y'-X')=0,$$
$$X''(Y'+Z')+Y''(X'+Z')-Z'''(X+Y)=0,$$
$$X'''(Y-Z)-Y''(X'+Z')-Z''(Y'-X')=0,$$
and so
$$M(Y-Z)X'=-X'Y'Z'(Y-Z)X''',$$
$$M(X-Z)Y'=-X'Y'Z'(X-Z)Y''',$$
$$M (X+Y)Z'=-X'Y'Z'(X+Y)Z'''.$$

Since   $X'Y'Z'\not=0$,  we deduce 
\begin{equation}\label{kk}
\frac{X'''}{X'}=\frac{Y'''}{Y'}=\frac{Z'''}{Z'}=K,
\end{equation}
where $K\in\r$ is a real constant. We solve these ODEs  according to the sign of $K$.

\begin{enumerate}
\item Case $K>0$. Let $K=k^2$, $k>0$. The solutions of (\ref{kk}) are
\begin{equation}\label{sol1}
\begin{split}
X(u)&=a_1+b_1 e^{k u}+c_1e^{-ku},\\
Y(v)&=a_2+b_2e^{kv}+c_2 e^{-kv},\\
Z(w)&=a_3+b_3e^{kw}+c_3 e^{-k w},
\end{split}
\end{equation}
where $a_i,b_i,c_i\in\r$, $1\leq i\leq 3$. These nine constants are not arbitrary because there is a relation between them thanks to (\ref{eq2}). We introduce the    functions $X$, $Y$ and $Z$ in (\ref{eq2}) and we replace $w$ by  $-u-v$ because $u+v+w=0$. Then (\ref{eq2}) is an equation of type
$$P_1e^{-ku}+P_2 e^{-kv}+P_3e^{ku}+P_3 e^{kv}+P_5e^{-ku-kv}+P_6e^{ku+kv}=0.$$
Since the exponential functions are  linearly independent,   the coefficients $P_i$ must be $0$. A computation of  $P_i$, $1\leq i\leq 6$, and setting to be $0$, yields
\begin{equation}\label{var1}
\left\{\begin{split}
(a_2-a_3)c_1+2b_2b_3&=0,\\
(a_1-a_3)c_2+2b_1b_3&=0,\\
(a_2-a_3)b_1+2c_2c_3&=0,\\
(a_1-a_3)b_2+2c_1c_3&=0,\\
(a_1+a_2)b_3+2c_1c_2&=0,\\
(a_1+a_2)c_3+2b_1b_2&=0.
\end{split}\right.
\end{equation}

\item Case $K<0$. Let $K=-k^2$, $k>0$. Then the solutions of (\ref{kk}) are 
\begin{equation}\label{sol2}
\begin{split}
X(u)&=a_1+b_1\cos(ku)+c_1\sin(ku),\\
Y(v)&=a_2+b_2\cos(kv)+c_2\sin(kv),\\
Z(w)&=a_3+b_3\cos(kw)+c_3\sin(kw).
\end{split}
\end{equation}
As in the previous case, and using that the trigonometric functions $\cos(ku)$, $\sin(ku)$, $ \cos(ku)$ and $\cos(kv)$ are  linearly independent, the relations between the constants $a_i$, $b_i$ and $c_i$  are now:
\begin{equation}\label{var2}
\left\{\begin{split}
(a_2-a_3)c_1-b_3c_2-b_2c_3&=0,\\
(a_1-a_3)c_2-b_3c_1-b_1c_3&=0,\\
(a_2-a_3)b_1+b_2b_3-c_2c_3&=0,\\
(a_1-a_3)b_2+b_1b_3-c_1c_3&=0,\\
(a_1+a_2)b_3+b_1b_2-c_1c_2&=0,\\
(a_1+a_2)c_3-b_2c_1-b_1c_2&=0.
\end{split}\right.
\end{equation}
\item Case $K=0$. Now the solutions of (\ref{kk}) are 
\begin{equation}\label{sol3}
\begin{split}
X(u)&=a_1+b_1u+c_1u^2,\\
Y(v)&=a_2+b_2v+c_2v^2,\\
Z(w)&=a_3+b_3w+c_3w^2.
\end{split}
\end{equation}
As in the case $K\not=0$, we insert the above solutions in (\ref{eq2}) and replace $w$ by $-u-v$, obtaining an equation of type
$$P_1+P_2u+P_3v+P_4uv+P_5u^2+P_6v^2+P_7u^2v+P_8uv^2=0.$$
The functions on $u$ and $v$ are  linearly independent, so all coefficients $P_i$ must vanish, $1\leq i\leq 8$. A computation of $P_i$ yields $P_7=P_8$ and $P_4$ is a linear combination of $P_5$ and $P_6$. Once computed the coefficients $P_i$ and setting to be $0$,   we deduce
\begin{equation}\label{var3}
\left\{\begin{split}
&a_1(b_2-b_3) +a_2(b_1-b_3)-a_3(b_1+b_2)=0,\\
&b_1b_2+b_2b_3+2c_1(a_2-a_3)+2c_3(a_1+a_2)=0,\\
&b_1b_2+b_1b_3+2c_2(a_1-a_3)+2c_3(a_1+a_2)=0,\\
&c_1(b_2+b_3)+c_3(b_1-b_2)=0,\\
&c_2(b_1+b_3)-c_3(b_1-b_2)=0,\\
&c_1c_2-c_1c_3-c_2c_3=0.
\end{split}\right.
\end{equation}
\end{enumerate}

Definitively, the functions $f$, $g$ and $h$ in (\ref{s1}) are the solutions of the ODEs of first order (\ref{sol1}), (\ref{sol2}) and (\ref{sol3}). We summarize in the following result the method of constructions of all separable ZMC surfaces of $\l^3$.

\begin{theorem} \label{t1}
Let $S$ be a separable ZMC  surface given by (\ref{s1}). Then the derivatives $f'$, $g'$ and $h'$   given in terms of  the functions $X$, $Y$ and $Z$  are the following:
\begin{enumerate}
\item Case $K\not=0$. Then the functions $X$, $Y$ and $Z$ are  given by expressions (\ref{sol1}) or  (\ref{sol2})  and the   constants $a_i$, $b_i$ and $c_i$ satisfy the relations (\ref{var1}) or (\ref{var2}) respectively.
\item Case $K=0$. Then the functions $X$, $Y$ and $Z$ are given by the expressions   (\ref{sol3}) and the constants $a_i$, $b_i$ and $c_i$ satisfy the relations (\ref{var3}).
\end{enumerate}
\end{theorem}

   In order to integrate  (\ref{sol1}), (\ref{sol2}) and (\ref{sol3}), we remark that the functions in the right-hand sides of these equations must be positive because the functions $X$, $Y$ and $Z$ are positive. In addition, we need to consider a  choice of the sign  for the square roots of the first derivative, namely, 
$$\int^u\frac{du}{\sqrt{X(u)}}=\pm x,\quad   \int^v\frac{dv}{\sqrt{Y(v)}}=\pm y,\quad\int^w\frac{dw}{\sqrt{Z(w)}}=\pm w,$$
where the $(+)$ or $(-)$ sign is chosen depending on whether the derivatives $f'(x)$, $g'(y)$ and $h'(z)$ are positive or negative. For simplicity,    we will usually choose the positive  sign.

We now prove that  varying the value of $k$ in (\ref{sol1}) or  (\ref{sol2}), the solution surface is the same up to a homothety. This will allow to fix the value of $k$ in (\ref{sol1}) or  (\ref{sol2}).

 \begin{proposition} \label{pr1}
 A change of $k$ in the solutions of  (\ref{sol1}) or  (\ref{sol2})  produces a homothetical ZMC surface which is also of separable type.
 \end{proposition}
 
 \begin{proof} Let $S$ be a separable ZMC surface given by $f(x)+g(y)+h(z)=0$ and suppose that it is a solution of (\ref{sol1}) or (\ref{sol2}) with $\lambda\not=0$. Define the functions 
$\tilde{f}(x)=\lambda f(x/\lambda)$, $\tilde{g}(y)=\lambda g(y/\lambda)$ and $\tilde{h}(z)=\lambda h(z/\lambda)$, defined in the intervals $\lambda I_1$, $\lambda I_2$ and $\lambda I_3$ respectively.  Then the separable  surface
$\tilde{S}=\{(\tilde{x},\tilde{y},\tilde{z}): \tilde{f}(\tilde{x})+\tilde{g}(\tilde{y})+\tilde{h}(\tilde{z})=0\}$   is the dilation of the surface $S$ by $\lambda$, in particular, $\tilde{S}$ is also a ZMC surface. 

With the utilized notation   in this section, we have  $\tilde{u}=\lambda u$, $\tilde{v}=\lambda v$ and $\tilde{w}=\lambda w$. Then $\tilde{X}(\tilde{u})=\tilde{f}'(\tilde{x})^2=f'(x/\lambda)^2=X(u)$, and similarly, $\tilde{Y}(\tilde{v})=Y(v)$ and $\tilde{Z}(\tilde{w})=Z(w)$. It is immediate
$$\tilde{X}'(\tilde{u})=\frac{d}{d\tilde{u}}\tilde{X}(\tilde{u})=\frac{1}{\lambda}X'(u),$$ 
and similarly  $\tilde{X}'''(\tilde{u})=X'''(u)/\lambda^3$. From (\ref{kk}) we deduce
$$\tilde{K}=\frac{\tilde{X}'''(\tilde{u})}{\tilde{X}'(\tilde{u})}=\frac{1}{\lambda^2}\frac{X'''(u)}{X'(u)}=\frac{1}{\lambda^2}K,$$
obtaining the result. 
\end{proof}

In the next sections   we will show  explicit examples of ZMC surfaces.  In some particular cases, the ODEs  can be solved by quadratures obtaining expressions in terms of   trigonometric or hyperbolic functions. In general,  we will see that  the class of ZMC surfaces of separable type can be expressed in terms of elliptic functions which are the inverses of certain elliptic integrals. 

\section{Separable ZMC  surfaces: case $K=0$.}\label{sec3}

In this section we will obtain particular examples of separable ZMC surfaces in  case $K=0$ of  Theorem \ref{t1}.   

Firstly, we will prove that the rotational ZMC surfaces belong to the case $K=0$, exactly  when   some of the constants $c_i$ in (\ref{var3}) are $0$.  Indeed, and without   loss of generality, we suppose $c_1=0$.  From the last equation of (\ref{var3}), we have $c_2c_3=0$. Because the arguments are symmetric on $c_2$ and $c_3$, we will assume in what follows that $c_2=0$. The fifth equation of (\ref{var3}) is   $c_3(b_1-b_2)=0$. Now there are two possibilities. 
\begin{enumerate}
\item Case    $c_3=0$. Then the second and third equations of (\ref{var3}) are $b_2(b_1+b_3)=0$ and $b_1(b_2+b_3)=0$, respectively. If $b_1=0$ (resp. $b_2=0$), then $X'=0$ (resp. $Y'=0$), which it is not possible. Hence $b_1=b_2=-b_3$ with $b_i\not=0$.  Now  the first equation of  (\ref{var3})  yields $a_3=a_1+a_2$. Then $X(u)=a_1+b_1u$, $Y(v)=a_2+b_1v$ and $Z(w)=a_3-b_1w$, in particular, $X+Y-Z=a_1+a_2-a_3=0$ proving that the surface is degenerate everywhere, which  is not possible.  

\item Case  $c_3\not=0$. Then $b_2=b_1\not=0$.  If we set $b_1=b$,  then $X(u)=a_1+bu$ and $Y(v)=a_2+bv$ and the integration leads to 
$$f(x)=\frac{b}{4}x^2-\frac{a_1}{b},\quad g(y)=\frac{b}{4}y^2-\frac{a_2}{b}.$$
Thus the implicit equation of the surface is 
$$x^2+y^2+\frac{4}{b}h(z)-4(a_1+a_2)=0,$$
proving  that the surface is a surface of revolution with respect to the $z$-axis. 
\end{enumerate}
As it was discussed in the Introduction, the rotational ZMC surfaces are classified, so we will discard this case, or equivalently, we will assume that $c_i\not=0$ for all $1\leq i\leq 3$.

\subsection{Example where all $a_i$ and $b_i$ are $0$}\label{s-31}

Let choose the constants as $a_i=b_i=0$, $1\leq i\leq 3$, in    (\ref{var3}). Then 
$$X(u)=c_1u^2,\quad Y(v)=c_2 v^2,\quad Z(w)=c_3 w^2,$$
where  (\ref{var3})  reduces into the single equation $c_1c_2-c_1c_3-c_2c_3=0$. Note that $c_i>0$, $1\leq i\leq 3$, because $X$, $Y$ and $Z$ are positive functions. The solutions are
$$f(x)=\pm e^{\sqrt{c_1}x},\quad  g(y)=\pm e^{\sqrt{c_2}y},\quad h(z)=\pm e^{\sqrt{c_3}z},$$
and the surface is 
$$\{(x,y,z)\in\l^3: \pm e^{\sqrt{c_1}x}   \pm e^{\sqrt{c_2}y} \pm e^{\sqrt{c_3}z}=0\}.$$
See Figure \ref{fig2}, left.  For the causal character of the surface, we study the sign of $X+Y-Z$. By using $w=-u-v$, 
$$X(u)+Y(v)-Z(w)=c_1u^2+c_2 v^2-c_3w^2=\frac{(c_1 u-c_2 v)^2}{c_1+c_2}.$$
Since $c_i>0$,   the surface is   timelike and it may be extended to regions of  lightlike points if   $c_1u-c_2v=0$.  If we choose opposite signs in the definition of the  functions $f$ and $g$, then   $c_1u-c_2v\not=0$ and the surface is always timelike. Now suppose   that   the functions have the same  sign, that is,  $f(x)=e^{\sqrt{c_1}x}$ and $g(y)=e^{\sqrt{c_2}y}$. Then the     surface is 
$$\{(x,y,z)\in\l^3: e^{\sqrt{c_1}x}   + e^{\sqrt{c_2}y} - e^{\sqrt{c_3}z}=0\}$$ 
and
$$c_1u-c_2v=c_1e^{\sqrt{c_1}x}-c_2e^{\sqrt{c_2}y}=e^{\sqrt{c_1}x+\log(c_1)}-e^{\sqrt{c_2}y+\log(c_2)}.$$
Thus $c_1u-c_2v=0$ is equivalent to $\sqrt{c_1}x+\log(c_1)=\sqrt{c_2}y+\log(c_2)$. Using  $w=-u-v$, we deduce  $\sqrt{c_3}z=\sqrt{c_1}x+\log(c_1+c_2)-\log(c_2)$, obtaining a straight-line. We conclude that   the surface can  be  extended to the set
$$\{(x,y,z): \sqrt{c_1}x+\log(c_1)=\sqrt{c_2}y+\log(c_2), \sqrt{c_3}z=\sqrt{c_1}x+\log(c_1+c_2)-\log(c_2)\}$$
which is a straight-line formed by lightlike points.

\subsection{Example where all $b_i$ are $0$ and all $a_i$ are not $0$: case 1}\label{s-32}

In this subsection, we will show two examples where the difference will be  that the constants $a_i$, $b_i$ and $c_i$ have  opposite signs.  
 For the first example, the  constants $a_i$, $b_i$ and $c_i$  are
$$b_1=b_2=b_3=0, \quad c_1=-\alpha^2\beta^2, c_2=-\alpha^2, c_3=-\beta^2$$
$$a_1=\alpha^2\beta^2,\quad a_2=\alpha^2\beta^4,\quad a_3=\alpha^4\beta^2,$$
 with $\alpha^2-\beta^2=1$ and $\beta\not=0$. Then it is immediate that they satisfy the equations (\ref{var3}) and the ODEs (\ref{sol3}) are

$$X(u)=\alpha^2\beta^2-\alpha^2\beta^2 u^2,\quad Y(v)=\alpha^2\beta^4-\alpha^2v^2,\quad Z(w)=\alpha^4\beta^2-\beta^2 w^2.$$
After integrating, we obtain
\begin{eqnarray*}
x&=&\int\frac{1}{\alpha\beta}\frac{du}{\sqrt{1-u^2}}=\frac{1}{\alpha\beta} \arcsin(u)\\
y&=&\int\frac{1}{\alpha}\frac{dv}{\sqrt{\beta^4- v^2}}=\frac{1}{\alpha} \arcsin(\frac{v}{\beta^2})\\
z&=&\int\frac{1}{\beta}\frac{dw}{\sqrt{\alpha^4-w^2}}=\frac{1}{\beta} \arcsin(\frac{w}{\alpha^2})
\end{eqnarray*}
Therefore
$$f(x)=\sin(\alpha\beta x),\quad g(y)=\beta^2\sin(\alpha y),\quad h(z)=\alpha^2\sin(\beta z)$$
and the surface is
$$\{(x,y,z)\in\l^3: \sin(\alpha\beta x)+\beta^2\sin(\alpha y)+\alpha^2\sin(\beta z)=0\}.$$
This surface is a triply periodic ZMC surface because the   functions $f$, $g$ and $h$ are periodic. See Figure \ref{fig1}, left.

We study the causal character of the surface. Letting $w=-u-v$, we have 
\begin{eqnarray*}
X(u)+Y(v)-Z(w)&=&-\alpha^2\beta^2u^2-\alpha^2v^2+\beta^2w^2=-\beta^2(\beta^2u^2+\frac{ v^2}{\beta^2}-2uv)\\
&=&-\beta^4(\sin(\alpha\beta x)-\sin(\alpha y))^2\leq 0.
\end{eqnarray*}
This implies that  the surface is spacelike except satisfy the equation $\sin(\alpha\beta x)-\sin(\alpha y)=0$, which are lightlike points.   The  region of  these  points is included in the set of  straight-lines of equations 
$$\{y=\ \beta x+2\pi\mathbb{Z}, z=-\alpha x+2\pi\mathbb{Z}\}\cup\{y=\pi-\beta x+2\pi\mathbb{Z},z=\alpha x+\pi+2\pi\mathbb{Z}\}.$$

The second example appears when we  reverse of sign all constant $a_i$ and $c_i$, that is,   
$$b_1=b_2=b_3=0, \quad c_1=\alpha^2\beta^2, c_2=\alpha^2, c_3=\beta^2$$
$$a_1=-\alpha^2\beta^2,\quad a_2=-\alpha^2\beta^4,\quad a_3=-\alpha^4\beta^2.$$
 Then   they satisfy (\ref{var3}) again and the  integrations lead to
$$f(x)=\pm\cosh(\alpha\beta x),\quad g(y)=\pm\beta^2\cosh(\alpha y),\quad h(z)=\pm\alpha^2\cosh(\beta z)$$
and the surface is
 $$\{(x,y,z)\in\l^3:\pm \cosh(\alpha\beta x)\pm \beta^2\cosh(\alpha y)\pm \alpha^2\cosh(\beta z)=0\}.$$
 There must be  two   opposite signs in the implicit equation of the surface because on the contrary the above set would be empty.  As in the above example, now   the surface is a timelike minimal surface. In order to study when the surface can be extended by lightlike points, without loss of generality, we suppose   that   the implicit equation of the surface is $\cosh(\alpha\beta x)+ \beta^2\cosh(\alpha y)- \alpha^2\cosh(\beta z)=0$. Then  
  $$0=\beta^2u^2+\frac{ v^2}{\beta^2}-2uv=\beta^2\left(\cosh(\alpha\beta x)-\cosh(\alpha y)\right)^2.$$
  Hence $y=\pm\beta x$. By using   $u+v+w=0$, we deduce $\cosh(\alpha x)=\cosh(z)$, so $z=\pm\alpha x$. Thus the surface can be extended to four straight-lines of lightlike points which meet at   the origin of coordinates.  See Figure \ref{fig1}, right. Here we point out that a ZMC surface containing a non-degenerate straight-line can be extended by symmetry by the reflection principle \cite{ac,kim}. However, we can not expect  symmetries along lightlike straight-lines. Recently, Akamine and Fujino have proved that   if two lightlike line segments do intersect,   a conelike singularity always appears at that intersection point, and hence there exists the reflection property  \cite{af}. This situation occurs for this surface: by the way, it did not appear in the surface of subsection \ref{s-31}  because there was only one straight-line.

 \begin{figure}[hbtp]
 \begin{center}\includegraphics[width=.3\textwidth]{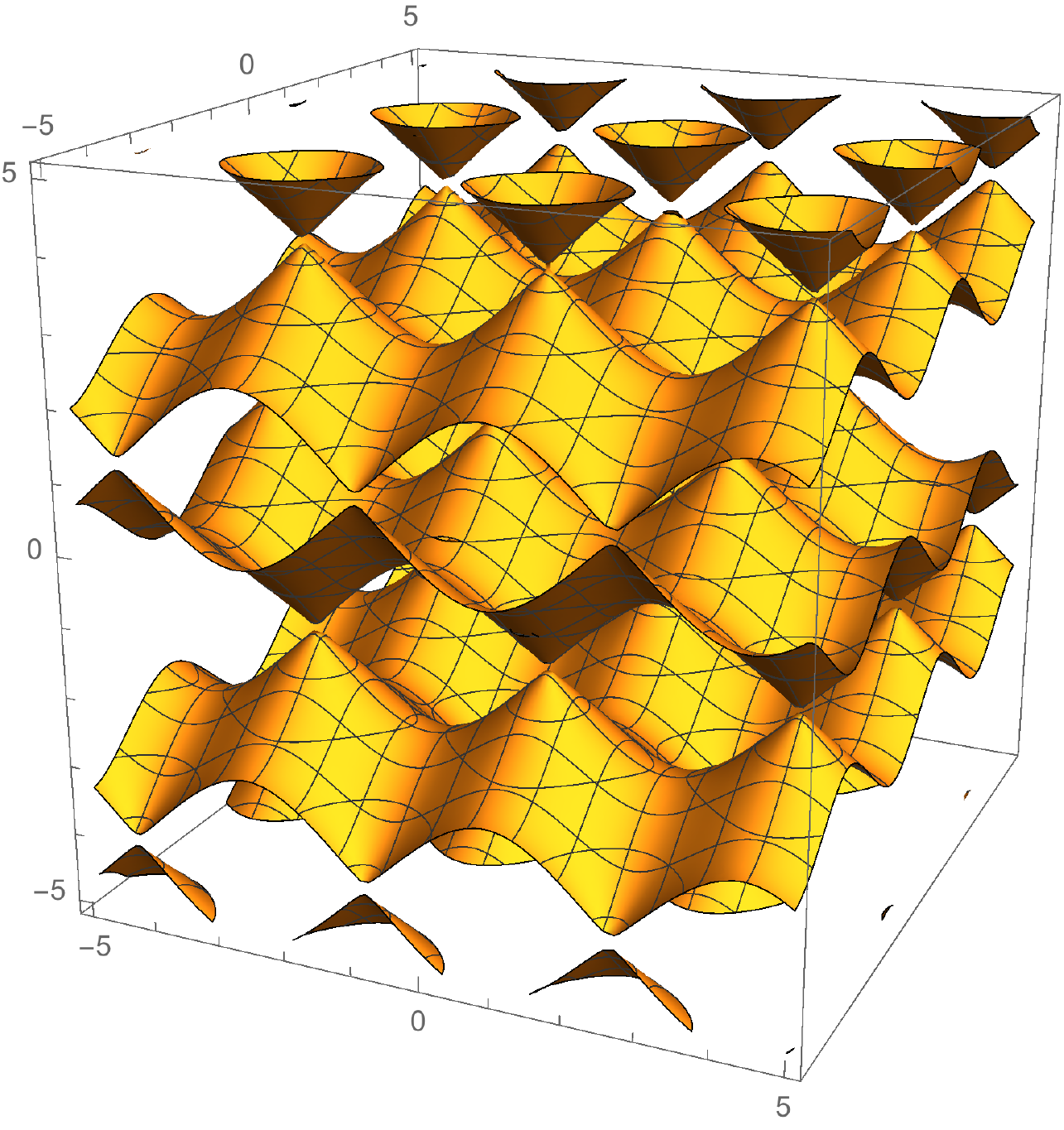}\quad \includegraphics[width=.3\textwidth]{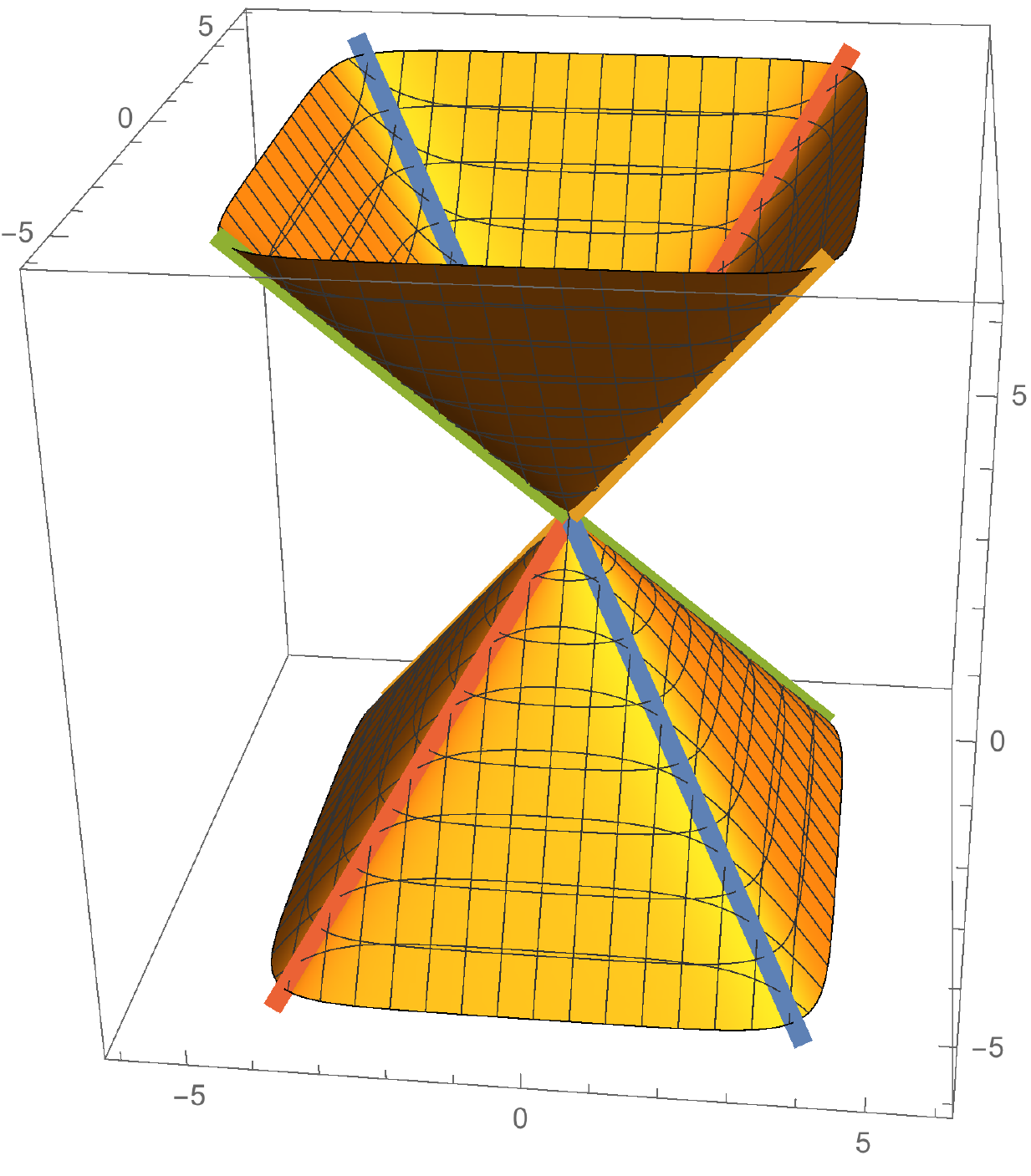}\end{center}
 \caption{Left: the surface $\sin(\alpha\beta x)+\beta^2\sin(\alpha y)+\alpha^2\sin(\beta z)=0$. Right: the surface $\cosh(\alpha\beta x)+\beta^2\cosh(\alpha y)-\alpha^2\cosh(\beta z)=0$, where we have indicated the four straight-lines of lightlike points}\label{fig1}
 \end{figure}

\subsection{Example where all $b_i$ are $0$ and all $a_i$ are not $0$: case 2}\label{s-33}

Let take the constant as  $b_1=b_2=b_3=0$,   $a_1=a_2=a_3/2=a$ and $c_1=c_2=2c_3=2c$, where  $a$ and $c$ are non-zero real numbers. Then the functions are
$$X(u)=a+2c u^2,\quad Y(v)=a+2c v^2,\quad Z(w)=2a+c w^2.$$
Since the functions $X$, $Y$ and $Z$ are positive, the case $a,c<0$ is not possible. The computation of the integral depends on the sign of $a$ and $c$, being  
$$\int\frac{dt}{\sqrt{m+nt^2}}=\left\{\begin{array}{ll}
\frac{1}{\sqrt{n}}\arcsinh\left(\sqrt{\frac{n}{m}}t\right)&m>0,n>0\\
\frac{1}{\sqrt{-n}}\arcsin\left(\sqrt{\frac{-n}{m}}t\right)&m>0,n<0\\
\frac{1}{\sqrt{n}}\arccosh\left(\sqrt{\frac{-n}{m}}t\right)&m<0,n>0.
\end{array}
\right.$$

If $a,c>0$, then the implicit equation of the surface is 
$$\sqrt{\frac{a}{2c}}\sinh(\sqrt{2c}x)+\sqrt{\frac{a}{2c}}\sinh(\sqrt{2c}y)+\sqrt{\frac{2a}{c}}\sinh(\sqrt{c}z)=0.$$
See Figure \ref{fig2}, right.
If $a>0$, $c<0$, then 

$$\sqrt{\frac{-a}{2c}}\sin(\sqrt{-2c}x)+\sqrt{\frac{-a}{2c}}\sin(\sqrt{-2c}y)+\sqrt{\frac{-2a}{c}}\sin(\sqrt{-c}z)=0.$$
 \begin{figure}[hbtp]
 \begin{center} \includegraphics[width=.3\textwidth]{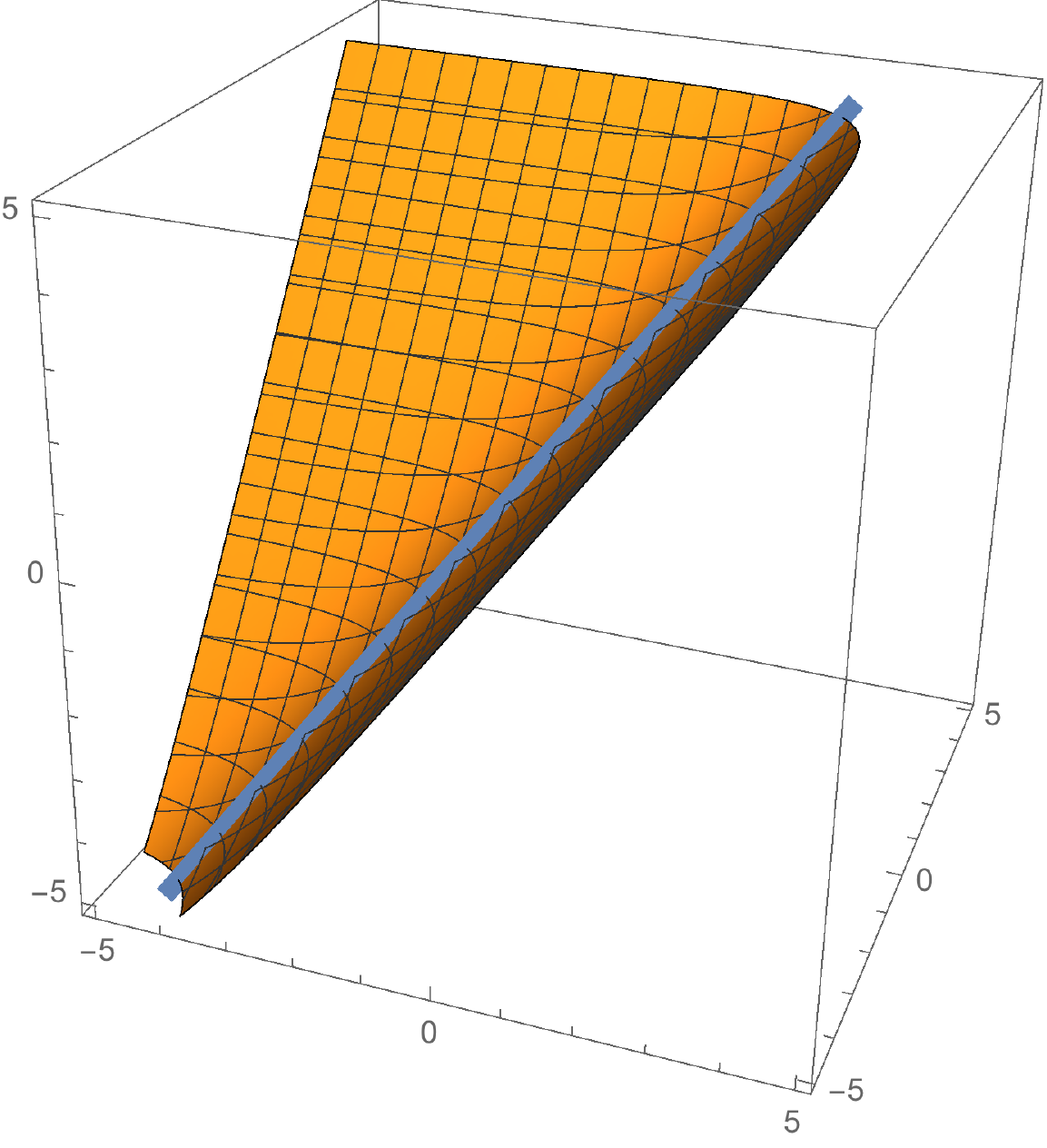}\quad \includegraphics[width=.3\textwidth]{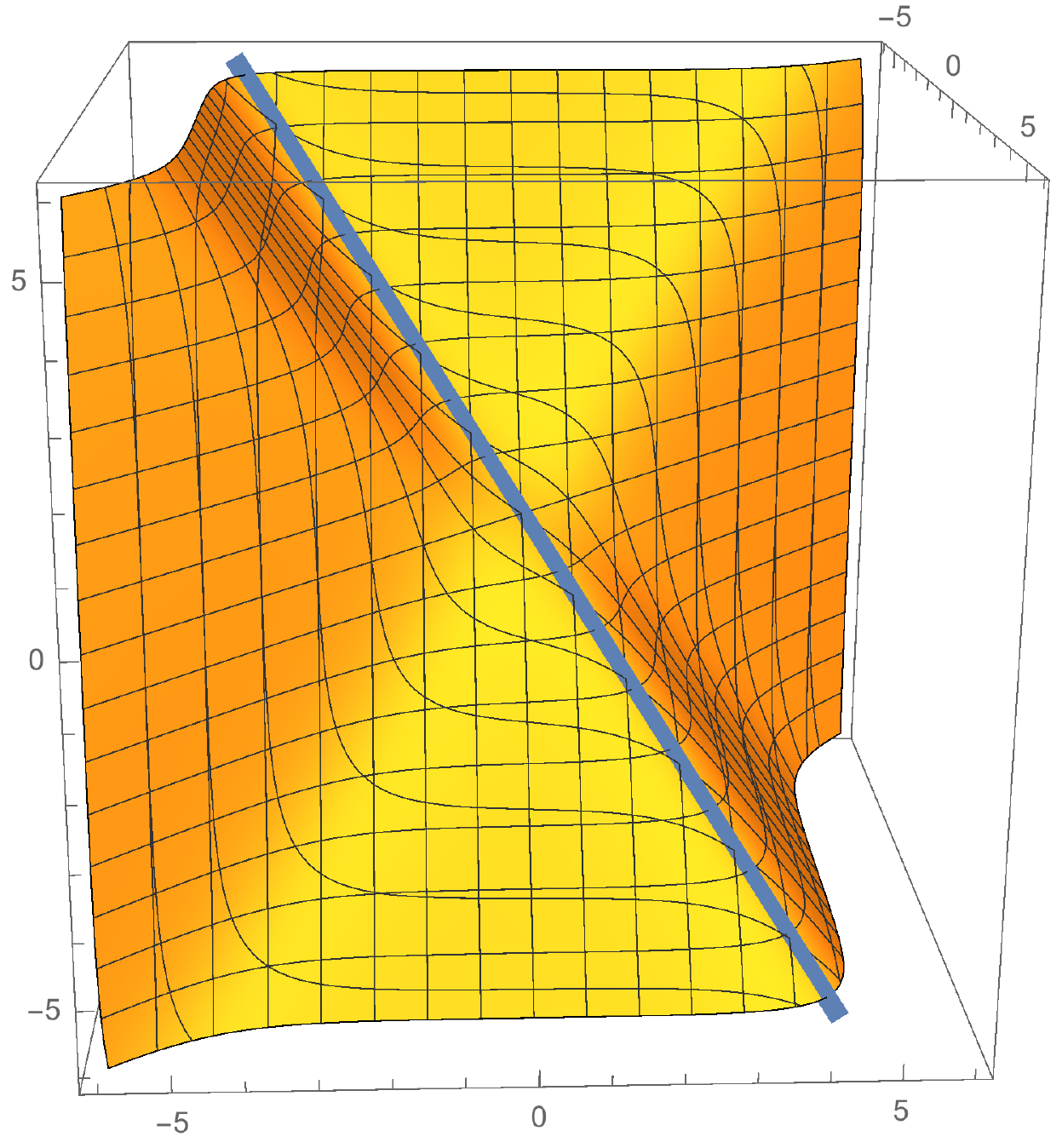} \end{center}
 \caption{Left: the surface  $e^x+e^y-e^{\sqrt{2}z/2}=0$ of subsection \ref{s-31}. Right: the surface  $\sinh(\sqrt{2}x)+\sinh(\sqrt{2}y)+2\sinh(z)=0$ of subsection \ref{s-33}. We have indicated  the straight-lines of lightlike points}\label{fig2}
 \end{figure}

If $a<0$, $c>0$, then 
$$\pm\sqrt{\frac{-a}{2c}}\cosh(\sqrt{2c}x)\pm\sqrt{\frac{-a}{2c}}\cosh(\sqrt{2c}y)\pm\sqrt{\frac{-2a}{c}}\cosh(\sqrt{c}z)=0.$$
By using the equation
$$X(u)+Y(v)-Z(w)=c(u-v)^2, $$
 we can say that the surface is spacelike or timelike if $c<0$ or $c>0$ respectively, except at the points   $\{u=v,w=-2v\}$. 
   
 Let us observe that the solution of the case $a>0$, $c<0$ and the case $a<0$, $c>0$ are similar to the surfaces   of subsection \ref{s-32}. In such a case,  we know  that  the surface   can be extended to lightlike points which are contained  in a set of straight-lines.
 
Finally, in the case   $a,c>0$, the surface is a timelike minimal surface except in the set $u=v$, $w=-2v$. This set is now $f(x)=g(y)$, $h(z)=-2g(y)$, or equivalently,    the straight-line   $\{x=y, z=-\sqrt{2}y\}$.

\section{Separable ZMC  surfaces: case $K>0$.}\label{sec4}

In this section we obtain particular examples of separable ZMC  surfaces  when $K> 0$ in   Theorem \ref{t1}. By   Proposition \ref{pr1}, we can assume that $k=2$ without  loss of generality.  To find explicit examples of separable ZMC surfaces, we follow the same procedure as  in the previous section. Firstly we determine    the real numbers $a_i$, $b_i$ and $c_i$ that satisfy (\ref{var1}), then    we   integrate  the differential equations (\ref{sol1}) and finally we will study the causal character of the surface. 

We divide this section in subsections according the numbers of the differential equations of (\ref{sol1}) that can  be solved by quadratures. The other solutions  will be expressed in terms of elliptic integrals.

\subsection{Case where the three differential equations are solved by quadratures}\label{s-41}

In this section we give four examples of ZMC surfaces of separable type where   all integrals in (\ref{sol1}) can be  solved by quadratures. 

\subsubsection{Example 1} 

Consider the following constants: 
$$ \begin{array}{lll}
  a_1=1,& b_1=0,&c_1=1\\
 a_2=1,& b_2=0,&c_2=m\\
 a_3=1, & b_3=-m, &c_3=0,
   \end{array}$$
 where $m\in\{-1,1\}$.  Then 
 $$X(u)=1+e^{-2u},\quad Y(v)=1+me^{-2v},\quad Z(w)=1-me^{2w}.$$
The integration of the first equation yields $f(x)=\log(\sinh(x))$. The integration of the other two differential equations depends on the sign of $m$: 
$$g(y)=\left\{\begin{array}{ll}
\log(\sinh(y))& m=1\\
\log(\cosh(y))& m=-1,\\
\end{array}
\right.\quad h(z)=\left\{\begin{array}{ll}
-\log( \cosh(z))& m=1\\
-\log( \sinh(z))& m=-1.\\
\end{array}
\right.$$
In each case of $m$, the implicit equation of the  surface is given by 
\begin{equation}\label{eq4111}
 \begin{split}\sinh(x)\sinh(y)=\cosh(z)& \mbox{ (case $m=1$)}\\
 \sinh(x)\cosh(y)=\sinh(z)& \mbox{ (case $m=-1$).}
\end{split}
\end{equation}
See Figure  \ref{fig3}.  The second surface in (\ref{eq4111}) is known as the timelike Scherk surface of second kind \cite{fu1} which is  an entire graph over the $xy$-plane.

For the causal character of the surface, 
 \begin{eqnarray*}
 X(u)+Y(v)-Z(w)&=&1+e^{-2u}+me^{-2v}+me^{-2u-2v}\\
 &=&(1+e^{-2u})(1+me^{-2v})=X(u)Y(v)>0.
 \end{eqnarray*}
 In case $m=1$, the surface is timelike and it can not be extended to regions of lightlike points. In case $m=-1$, the surface is timelike again and also it can be extended to regions of lightlike points when $g'(y)=0$, which is equivalent to $y=0$.  From the equation of the surface, we deduce $\sinh(x)=\sinh(z)$, so we have $x=z$. Thus the lightlike points consists in the straight-line $\{y=0,x=z\}$. Therefore the surface is a timelike minimal surface that can be extended to one lightlike straight-line but the surface does not change type across this line.    
 
  \begin{figure}[hbtp]
\begin{center}\includegraphics[width=.3\textwidth]{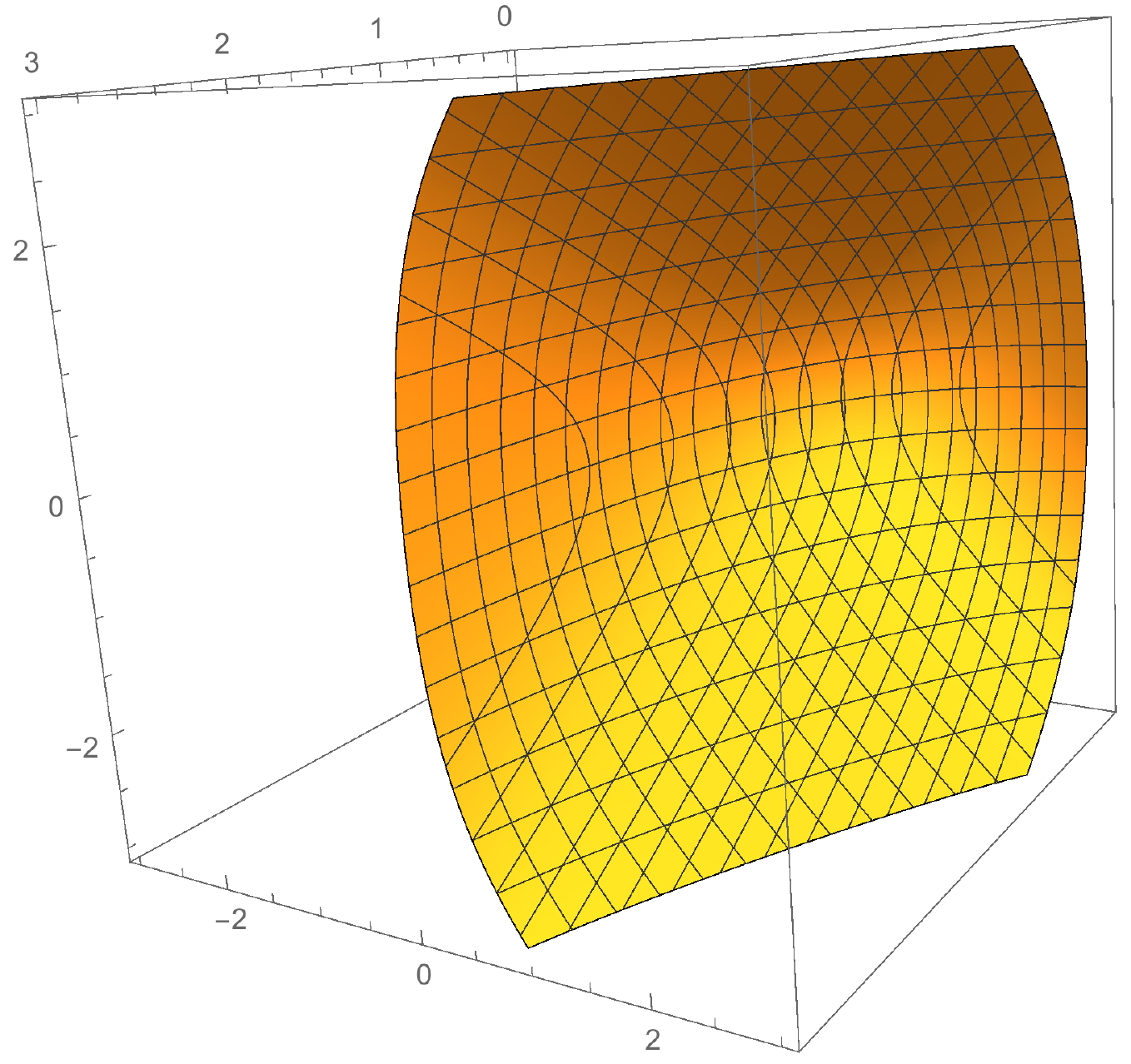}\quad \includegraphics[width=.3\textwidth]{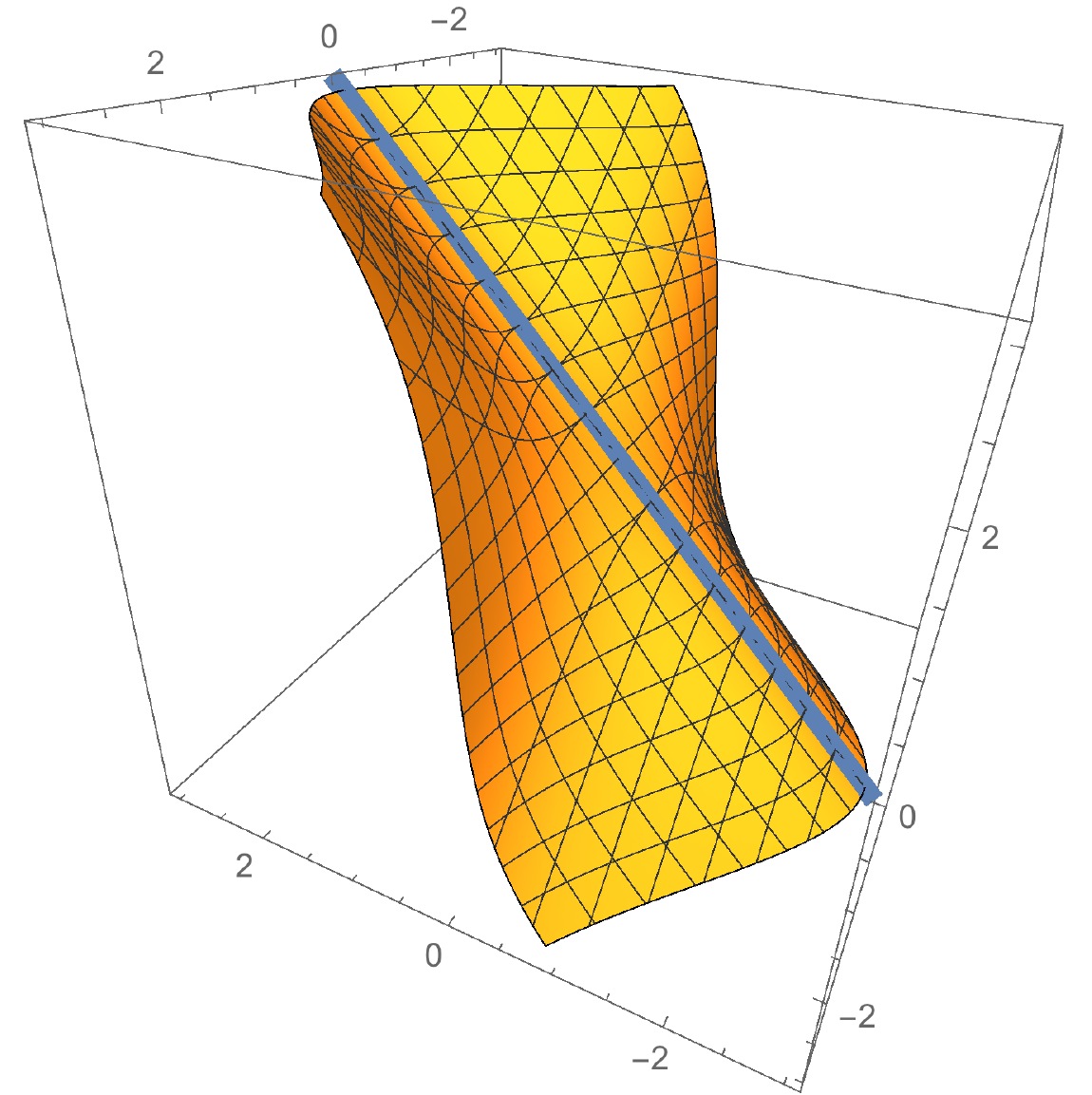}\end{center}
 \caption{Left: the surface  $ \sinh(x)\sinh(y)=\cosh(z)$. Right: the surface $\sinh(x)\cosh(y)=\sinh(z)$  }\label{fig3}
 \end{figure}

 \subsubsection{Example: the Scherk surfaces} 
Consider the following constants: 
$$\begin{array}{lll}
 a_1=1, & b_1=0, &c_1=-1\\
  a_2=1, & b_2=0,& c_2=m\\
 a_3=1, & b_3=m,& c_3=0,
   \end{array}$$
 where $m\in\{-1,1\}$.  Then 
 $$X(u)=1-e^{-2u},\quad Y(v)=1+me^{-2v},\quad Z(w)=1+me^{2w}.$$
The integration yields $f(x)=\log(\cosh(x))$ and
$$g(y)=\left\{\begin{array}{ll}
\log(\sinh(y))& m=1\\
\log(\cosh(y))& m=-1\\
\end{array}
\right.\quad h(z)=\left\{\begin{array}{ll}
-\log( \sinh(z))& m=1\\
-\log( \cosh(z))& m=-1.\\
\end{array}
\right.$$
For $m=1$,  the surface is $\cosh(x)\sinh(y)=\sinh(z)$, which appeared in the above subsection interchanging the roles of the variables $x$ and $y$. 

For $m=-1$, the implicit equation of surface is  $ \cosh(x)\cosh(y)=\cosh(z)$, see Figure  \ref{fig4}, left. This surface is called the timelike Scherk surface of first kind \cite{fu1}. The causal character of the surface is given by the sign of the function
 $$X(u)+Y(v)-Z(w)=(1-e^{-2u})(1+me^{-2v})=f'(x)^2g'(y)^2,$$
 and it says that   the surface is timelike. We study the  extension of the surface to regions of lightlike points. This occurs if $f'(x)=0$ or $g'(y)=0$, that is,   $\sinh(x)=0$ or $\sinh(y)=0$. We conclude that this region is formed by four straight-lines of equations $\{x=0,z=\pm y\}$ and $\{y=0, z=\pm x\}$. Here we have a conelike point at the origin according \cite{af}.
 
 If $m=-1$,   it  is possible to obtain new examples of separable ZMC surfaces by changing the signs of the constants $a_i$, $b_i$ and $c_i$ (this is not possible to do it for $m=1$ because the functions $Y$ and $Z$ would be negative). Now
$$\begin{array}{lll}
 a_1=-1, & b_1=0,&c_1=1\\
 a_2=-1, &b_2=0, &c_2=1\\
 a_3=-1, & b_3=1,& c_3=0.
   \end{array}$$
The integration of the three differential equations gives the surface as 
 $$\{(x,y,z)\in\l^3: \sin(x)\sin(y)=\sin(z)\}.$$
  This surface is known in the literature as the spacelike Scherk surface \cite{fu1}, see Figure \ref{fig4}, right. Since the sign of  the function $X+Y-Z$ has  changed, the surface is now spacelike and can be extended to lightlike points when $f'(x)g'(y)=0$. This occurs  when $\cos(x)=0$ or $\cos(y)=0$. Thus the set of lightlike points is  $\{x=\pi/2+2\pi\mathbb{Z}, y=z+2\pi\mathbb{Z}\}\cup\{x=\pi-z+2\pi\mathbb{Z}, y=\pi/2+2\pi\mathbb{Z}\}$. This maximal surface is a triply periodic surface and belongs to a family of triply periodic maximal surfaces which contains as particular examples, the H-type Schwarz surface and the D-type maximal surface \cite{fu4}. Moreover, it is also the Scherk saddle tower under the motion of the Wick rotation of  $\e^3$ \cite{as}.


\begin{figure}
\begin{center}
\includegraphics[width=.3\textwidth]{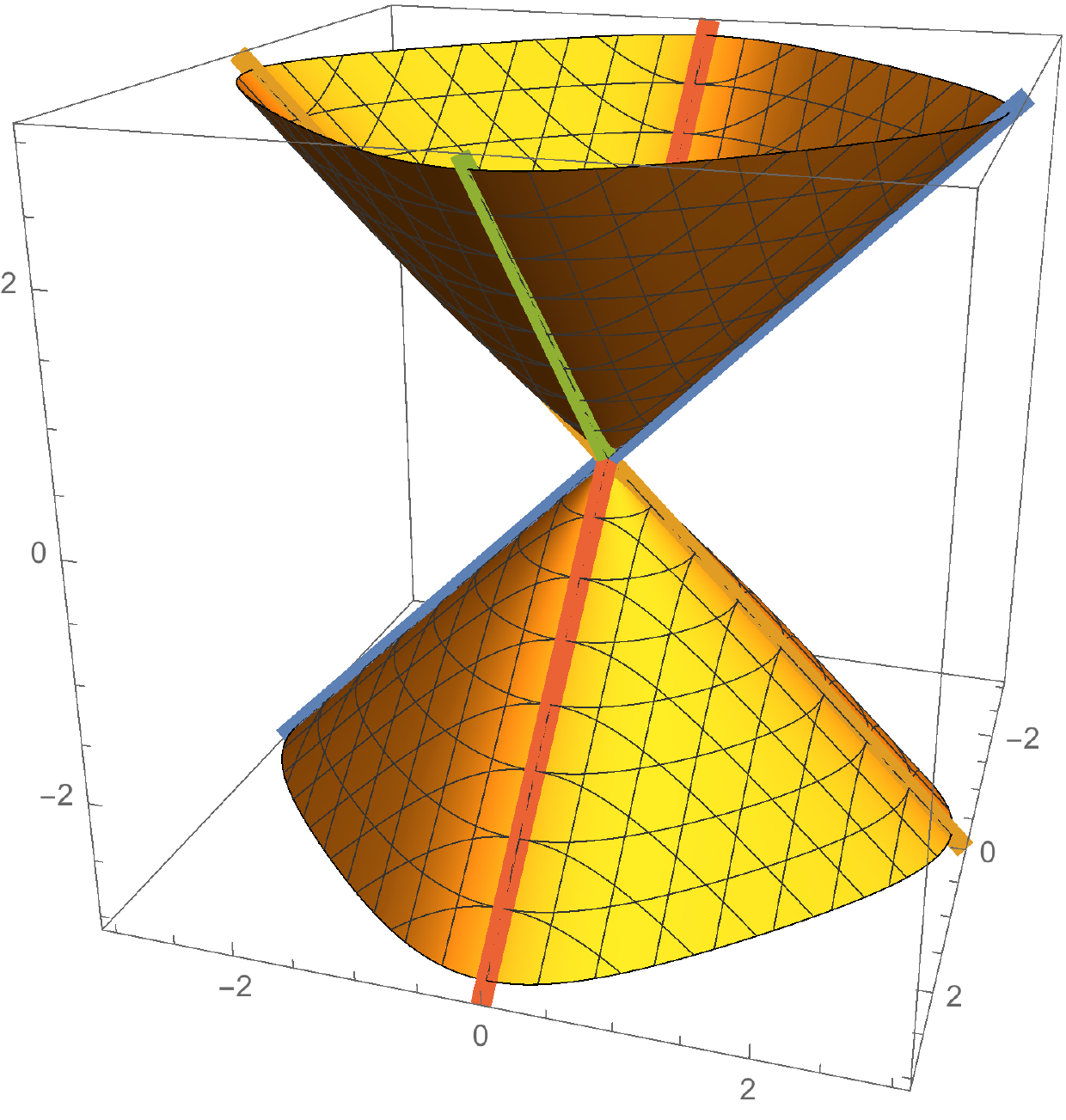}\quad \includegraphics[width=.3\textwidth]{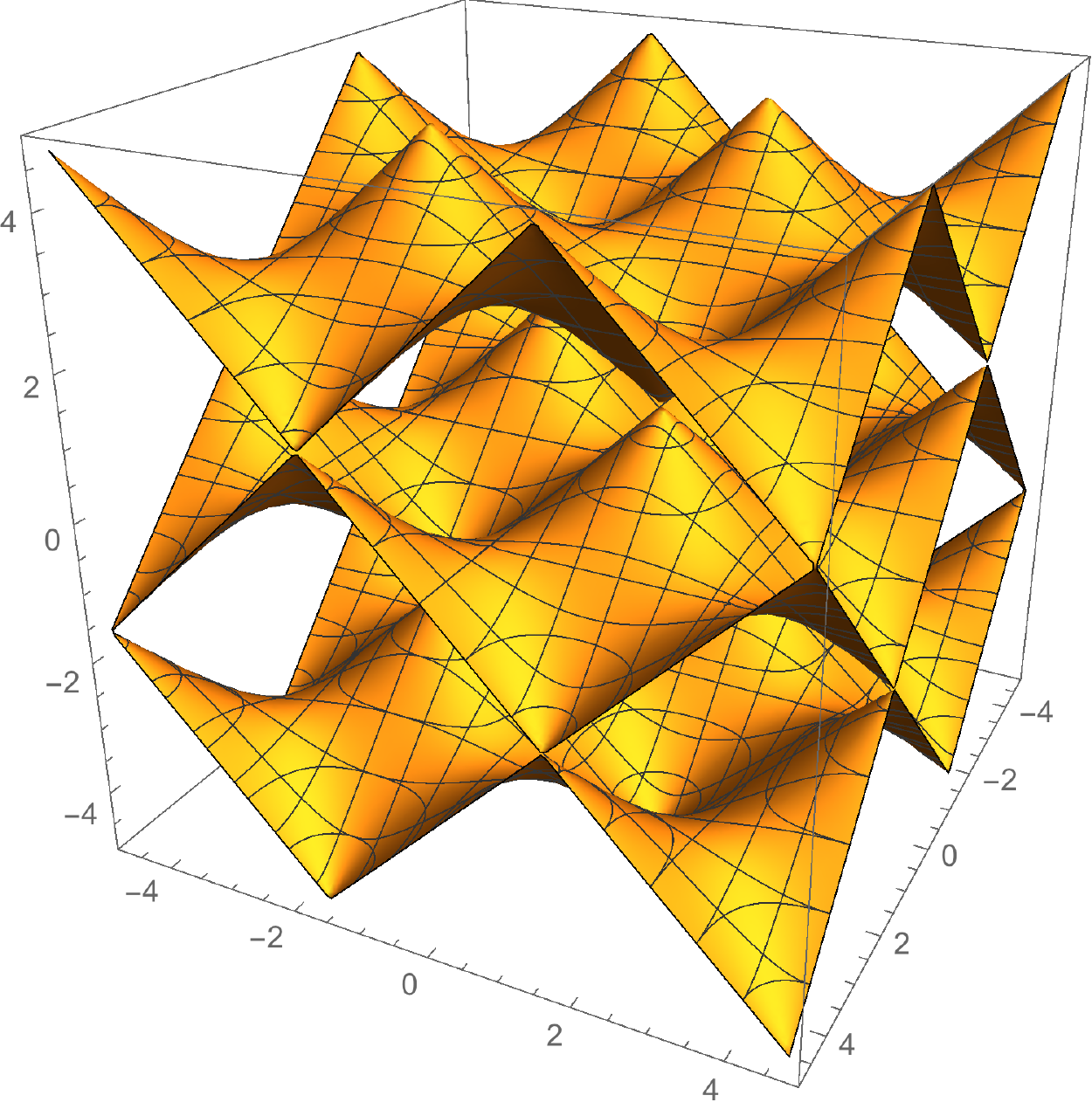}
\end{center}
 \caption{Left:  the surface $ \cosh(x)\cosh(y)=\cosh(z)$. Right: the surface  $\sin(x)\sin(y)=\sin(z)$}\label{fig4}
 \end{figure}
 
\subsubsection{Example 3: helicoids}
 A right helicoid in $\l^3$ is the surface obtained by moving a straight-line contained in a plane by means of a uniparametric group of skew motions of $\l^3$ around an axis contained in the plane (\cite{bkp}). The right helicoids have zero mean curvature. In this subsection we will obtain the right helicoids whose axis is spacelike or timelike. These surfaces   have regions of the three types of causal character.  Firstly, consider the following constants
$$\begin{array}{lll}
  a_1=0,&b_1=1,& c_1=0\\
 a_2=0,&b_2=0 , & c_2=1\\
  a_3=2,& b_3=1, &c_3=1.
   \end{array}$$
   Then the functions are
 $$X(u)=e^{2u},\quad Y(v)=e^{-2v},\quad Z(w)=2+e^{2w}+e^{-2w}.$$
The integration yields 
$$f(x)=-\log(-x),\quad g(y)=\log(y),\quad h(z)=\log(\tan(z)),$$
and the surface writes as $x=-y\tan(z)$. It is the right helicoid whose axis is  the   $z$-axis (timelike axis) and also known as the   elliptic helicoid (\cite{kim}) or the helicoid  of the first kind (\cite{ko}). 

Other choice of the constants is the following: 
$$\begin{array}{lll}
 a_1=0, &b_1=1,& c_1=0\\
 a_2=-2, & b_2=1 ,& c_2=1\\
   a_3=0, & b_3=0,&c_3=1.
   \end{array}$$
   Then 
 $$X(u)=e^{2u},\quad Y(v)=-2+e^{2v}+e^{-2v},\quad Z(w)= e^{-2w}.$$
The integration  of the three differential equations yields 
$$f(x)=-\log(-x),\quad g(y)=\log(-\tanh(y)),\quad h(z)=\log(z),$$
and the implicit equation of the surface is $x=z\tanh(y)$. This surface is the right helicoid whose axis is   the $y$-axis (spacelike axis) and  known as the hyperbolic helicoid (\cite{kim}) or the helicoid  of the second kind (\cite{ko}). Each helicoid is transformed into the other one by means of the Wick rotation \cite{as}.

\subsubsection{Example 4} 
Consider the following constants: 
$$\begin{array}{lll}
 a_1=-2, & b_1=1,& c_1=1\\
  a_2=-2, & b_2=1,&c_2=1\\
  a_3=-1, &b_3=1/2 ,& c_3=\frac12.
   \end{array}$$
 Then 
 $$X(u)=-2+e^{2u}+e^{-2u},\quad Y(v)=-2+e^{2v}+e^{-2v},\quad Z(w)=-1+\frac12 e^{2w}+\frac12e^{-2w}.$$
The integration of the first differential equation is
$$x=\int \frac{df}{\sqrt{-2+e^{2f}+e^{-2f}}}=\int^{e^f}\frac{d\tau}{\sqrt{(1-\tau^2)^2}},$$
and similarly for the second one. For the third equation, we have 
$$z=\int \frac{dg}{\sqrt{-1+e^{2f}/2+e^{-2f}/2}}=\int^{e^h}\frac{\sqrt{2}d\tau}{\sqrt{(1-\tau^2)^2}}.$$
The integration  by quadratures depends on the case of $\tau^2<1$ or $\tau^2>1$. If $\tau^2<1$ in the first equation, that is, if $f<0$, then  $f(x)=\log\tanh(x)$ and if $\tau^2>1$ ($f>0$), then 
$f(x)=-\log\tanh(x)$. Since the   functions $f$, $g$ and $h$ can not have the same sign, without loss of generality, we suppose that $f,g<0$ and $h>0$. In such a case,  $g(y)=\log(\tanh(y))$ and  $h(z)=-\log(\tanh(z/\sqrt{2}))$. Thus the   surface is
$$\{(x,y,z)\in\l^3: \tanh(x)\tanh(y)=\tanh(z/\sqrt{2})\}.$$
This surface appeared in  \cite{st} and it is  of mixed type: see Figure \ref{fig5}, left.

\subsection{Case where two differential equations are solved by quadratures}

We will give three examples where two differential equations in (\ref{sol1}) can be solved by quadratures and the solution of the other differential equation is given in terms of elliptic integrals.
\subsubsection{Example  1} \label{s-421}
 
 Consider the following constants that satisfy (\ref{var1}):
$$\begin{array}{lll}
 a_1=1, & b_1=0 ,& c_1=1\\
   a_2=-1, & b_2=1,& c_2=0\\
 a_3=2m-1, & b_3=m,&c_3=m-1,
   \end{array}$$
   
 where $m$ is a real parameter.  The    functions $X$, $Y$ and  $Z$ are 
 $$X(u)=1+e^{-2u},\quad Y(v)=-1+e^{2v},\quad Z(w)=2m-1+m e^{2w}+(m-1)e^{-2w}.$$
 Notice that $Z(w)$ is a positive function, so the value of the parameter $m$ is not arbitrary: for example, if $m=0$, then $Z(w)=-1-e^{-2w}<0$, which is not possible. The solutions of the first two differential equations are
$$f(x)=\log(\sinh(x)),\quad g(y)=- \log(\sin(y)).$$

For the third equation,   
$$\int^{e^h}\frac{d\tau}{\sqrt{m \tau^4+(2m-1)\tau^2+m-1}}=\pm z,$$

or 
\begin{equation}\label{eq-z}
\int^{e^h}\frac{d\tau}{\sqrt{(\tau^2+1)(m\tau^2-1+m)}}=\pm z.
\end{equation}
This integral is elliptic and can not be solved by quadratures in general. We show two examples by taking particular values of $m$.  
\begin{enumerate}
\item Case $m=1/2$. Then  $Z(w)=\sinh(2 w)$  and (\ref{eq-z}) is
\begin{equation}\label{t4}
 \xi:=\sqrt{2}\int^{e^{h}}\frac{1}{\sqrt{\tau^4-1}}d\tau=z.
 \end{equation}
Let $V$ the inverse of the function $\xi$. Then  
$$h(z)=\log\left(V(\frac{z}{\sqrt{2}})\right)$$
and the      surface is
$$\{(x,y,z)\in\l^3: \sinh(x) V(\frac{z}{\sqrt{2}})=  \sin(y)\}.$$
On the other hand, 
\begin{eqnarray*} X(u)+Y(v)-Z(w)&=&e^{-2u}+e^{2v}-\frac{e^{2w}-e^{-2w}}{2}\\
&=&\frac12 e^{-2u-2v}\frac{\cosh(x)^4-\cos(y)^4}{\sin(y)^4}>0
\end{eqnarray*}
and the surface is timelike at every point, except when $\cosh(x)=\cos^2(y)=1$. However, there are not points with $\cosh(x)=1$ and $\sin(y)=0$ because of the definitions of the functions $f$ and $g$. Thus the surface is timelike and can not be extended to lightlike points.

\item Case $m=1$. Now the integral (\ref{eq-z}) can be solved by quadratures, exactly,   $h(z)=-\log(\sinh(z))$. The surface is 
$$\{(x,y,z)\in\l^3: \sinh(x)=\sin(y)\sinh(z)\}.$$
This surface is singly periodic along the $y$-direction by the periodicity of $\sin(y)$, see Figure \ref{fig5}, middle. 
We study the causal character of the surface. Since $w=-u-v$, we have 
$$X(u)+Y(v)-Z(w)=e^{-2u}+e^{2v}-e^{2w}-1=(1+e^{-2u-2v})(e^{2v}-1)> 0$$
because $e^{2v}-1=g'(y)^2>0$. Then the surface is timelike. At the set of points where $g'(y)=0$, that is $\cos(y)=0$, the surface can be extended into a region of lightlike points. By the equation of the surface, $\sinh(x)=\sinh(z)$ or $\sinh(x)=-\sinh(z)$. This set is formed by  the straight-lines  $\{ x=z, y=\pi/2+2\pi\mathbb{Z}\}\cup \{ x=-z, y=-\pi/2+2\pi\mathbb{Z}\}$.  Let us observe that this surface is a singly periodic surface along the $y$-axis and it is the Wick rotation of the timelike Scherk surface of second kind that appeared in (\ref{eq4111}): see details in \cite{as}.  
\end{enumerate}

 \begin{figure}[hbtp]
\begin{center}\includegraphics[width=.3\textwidth]{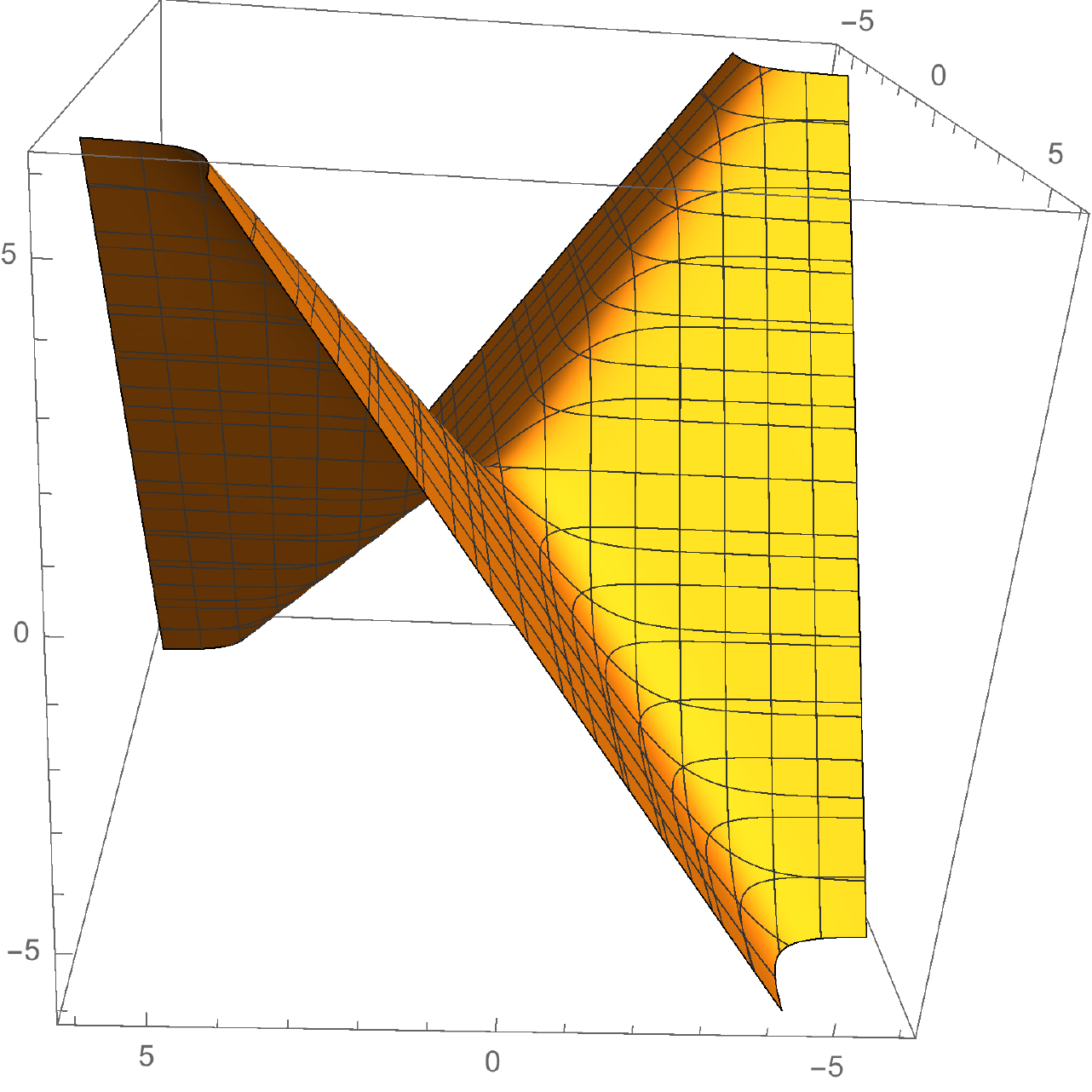}\quad \includegraphics[width=.3\textwidth]{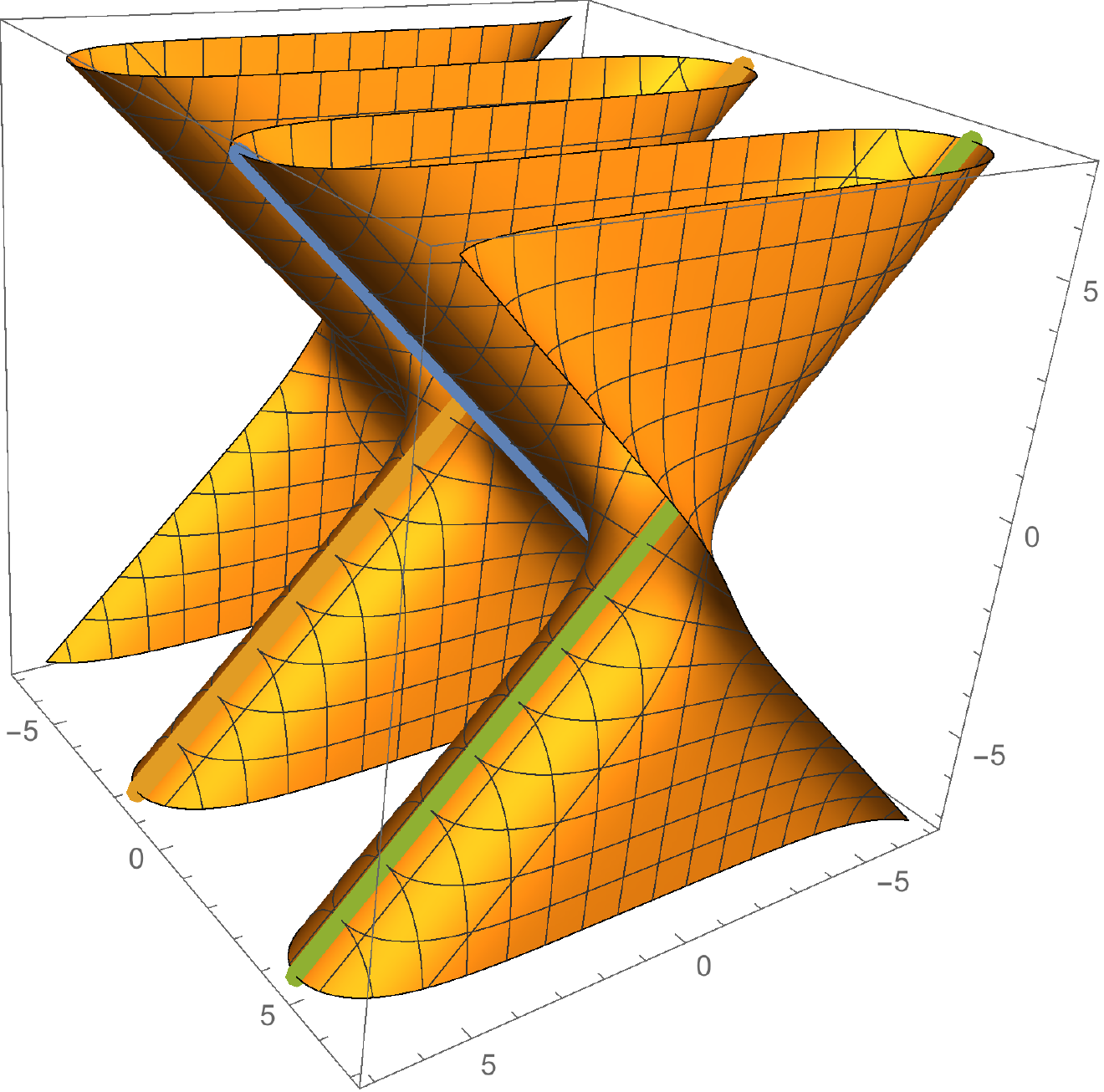}\quad \includegraphics[width=.3\textwidth]{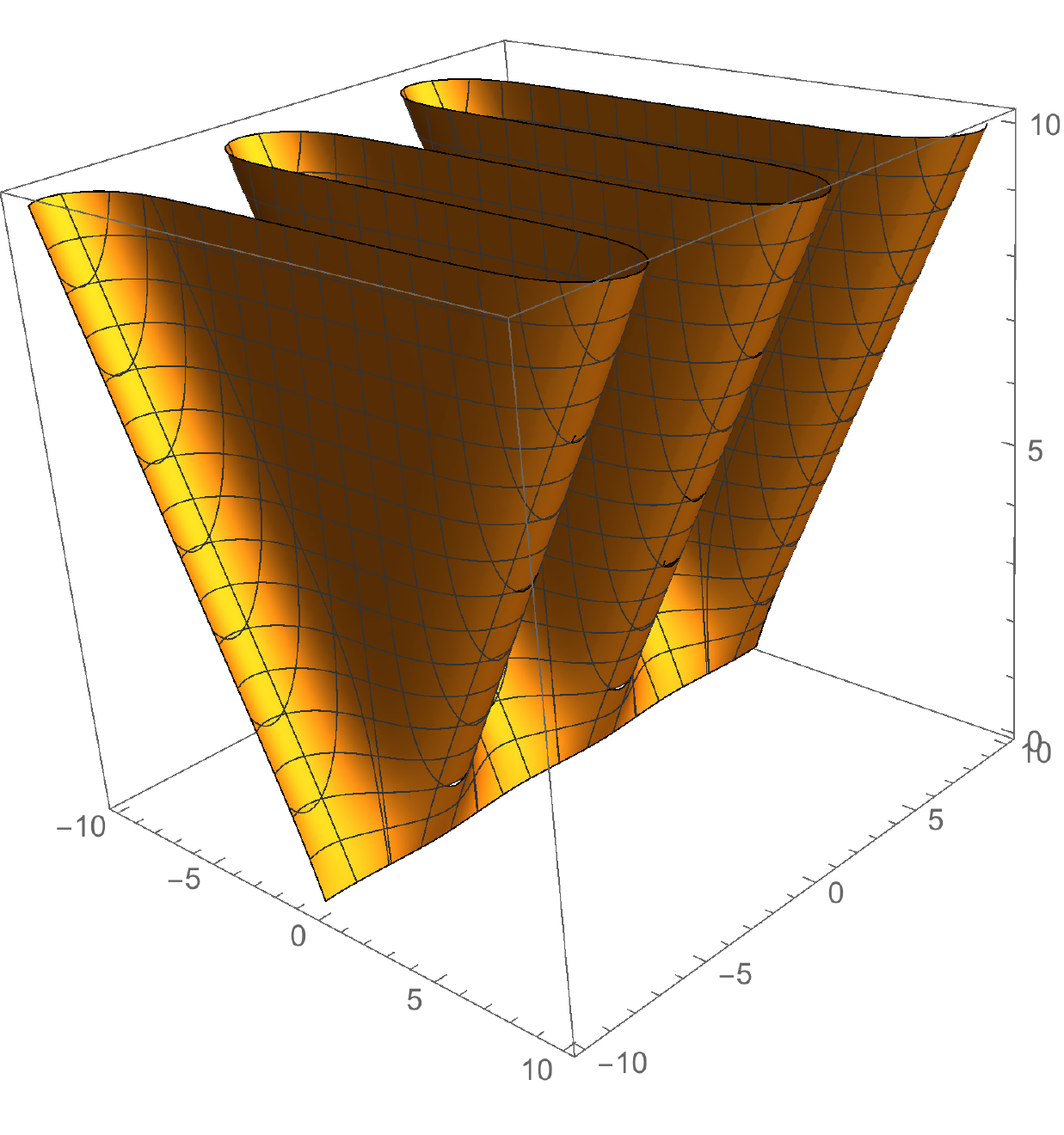}\end{center}
 \caption{Left: the surface $\tanh(x)\tanh(y)=\tanh(z/\sqrt{2})$. Middle: the surface $\sinh(x)=\sin(y)\sinh(z)$. Right: the surface $M(y/\sqrt{2})=\sinh(x)/\sinh(z)$}\label{fig5}
 \end{figure}

\subsubsection{Example  2} 

 Consider the following constants: 
$$\begin{array}{lll}
 a_1=1,& b_1=1,& c_1=0\\
 a_2=1-2m, & b_2=m-1,&c_2=m\\
  a_3=1,&b_3=0,&c_3=1,
   \end{array}$$
 where $m$ is a real parameter. The functions $X$, $Y$ and $Z$ are   
$$X(u)=1+e^{2u},\quad Y(v)=1-2m+(m-1)e^{2v}+me^{-2v},\quad Z(w)=1+e^{-2w}.$$
The integrations of the first and third differential equation yield 
$$f(x)=-\log(\sinh(x)),\quad h(z)=\log(\sinh(z)). $$
For the function $g$, we have the elliptic integral
\begin{equation}\label{eq-422}
\pm y=\int^{e^g}\frac{d\tau}{\sqrt{(m-1) \tau^4+(1-2m)\tau^2+m}}=\int^{e^g}\frac{d\tau}{\sqrt{(\tau^2-1)((m-1)\tau^2-m)}}.
\end{equation}

\begin{enumerate}
\item Case $m=0$. Then   $g(y)=-\log(\cosh(y))$ and the implicit equation of the surface is 
$\sinh(x)\cosh(y)=\sinh(z)$ which appeared in the subsection \ref{s-41}. 
\item Case  $m=1$. Then  the surface   is  $\sinh(x)=\sin(y)\sinh(z)$, which appeared again in the above subsection.   
\item Case $m=1/2$. The integral (\ref{eq-422})  is now
\begin{equation}\label{eqt4}
\psi:=\int^{e^g}\frac{d\tau}{\sqrt{1-\tau^4}}=\frac{y}{\sqrt{2}}.
\end{equation}
Let $M(\psi)$ be the inverse of the function $\psi$. Then  
$$g(y)=\log\left(M(\frac{y}{\sqrt{2}})\right),$$
and the surface is 
$$\{(x,y,z)\in\l^3: \sinh(z)M(\frac{y}{\sqrt{2}})= \sinh(x)\}.$$
This surface is singly periodic along the $y$-axis, see Figure \ref{fig5}, right. For the causal character, we determine the sign of the function
\begin{eqnarray*}
X(u)+Y(v)-Z(w)&=&\frac12e^{-2u-2v}\left(e^{4u}(e^{2w}+1)^2-(e^{2u}+1)^2\right)\\
&=&\frac12e^{-2u-2v}\frac{\cosh(z)^4-\cosh(x)^4}{\sinh(x)^4},
\end{eqnarray*}
hence the surface is of mixed type. The lightlike points is the set of points such that $\cosh(x)=\cosh(z)$, that is, $\{(x=z, M(y/2)=1\}\cup  \{(x=-z, M(y/2)=1\}$ up to periodicity. By the function $M$, this set of points is formed by straight-lines.  
\end{enumerate}

\subsubsection{Example  3} 
We show two new examples of separable ZMC surfaces with similar choices of the constants. The first example corresponds with the choice of constants 
$$\begin{array}{lll}
 a_1=-1, &b_1=0,& c_1=1\\
 a_2=1, &b_2=-1,& c_2=0\\
 a_3=m, & b_3=\frac{1-m}{2},& c_3=\frac{-1-m}{2},
   \end{array}$$
 where $m$ is a real parameter. The functions $X$, $Y$ and $Z$ are   
$$X(u)=-1+e^{-2u},\quad Y(v)=1-e^{2v},\quad Z(w)=m+\frac{1-m}{2}e^{2w}-\frac{1+m}{2}e^{-2w}.$$
The integration of the first two differential equations yields 
$$f(x)=\log(\sin(x)),\quad g(y)=-\log(\cosh(y)). $$
For the function $h$, we have
\begin{equation}\label{eq-423}
\int^{e^h}\frac{d\tau}{\sqrt{(1-m) \tau^4+2m\tau^2-1-m}}=\int^{e^h}\frac{d\tau}{\sqrt{(\tau^2-1)((m-1)\tau^2-m-1)}}=\frac{z}{\sqrt{2}}.
\end{equation}

In general the integral (\ref{eq-423}) is elliptic. We show two particular examples.
\begin{enumerate}
\item Case $m=1$. Then the integration of (\ref{eq-423}) yields$h(z)=\log(\cosh(z))$ and the   surface is 
$$\{(x,y,z)\in\l^3:\sin(x)\cosh(z)=\cosh(y)\}.$$
 Notice that the values of $x$ such that $\sin(x)=0$ are not in the domain of the surface, so the surface is not singly periodic along the $x$-axis: see Figure \ref{fig6}, left. In this case, 
$$X(u)+Y(v)-Z(w)=(1-e^{-2u})(e^{-2w}-1)> 0$$ 
and the surface is timelike. The surface extends to lightlike points in the set $f'(x)h'(z)=0$, that is, $\cos(x)=0$ or $\sinh(z)=0$. Without loss of generality,  we assume that the domain of $f$ is $(0,\pi)$, then $\cos(x)=0$ yields $x=\pi/2$, so the set of lightlike points is formed by two straight-lines, namely, $\{x= \pi/2, z=\pm y\}$. If $\sinh(z)=0$, then $z=0$ and from the equation of the surface,  $\sin(x)=\cosh(y)$, hence $x=\pi/2$, $y=0$, showing that the point $(\pi/2,0,0)$ is a singularity of the surface. Furthermore the surface reflects along this singularity (\cite{af}).
\item Case  $m=0$. Then the integral (\ref{eq-423}) is
$$\int^{e^h}\frac{d\tau}{\sqrt{  \tau^4-1}}=\frac{z}{\sqrt{2}}.$$
This  integral has appeared in (\ref{t4}). Then the   surface is 
$$\{(x,y,z)\in\l^3: \sin(x)V(\frac{z}{\sqrt{2}})= \cosh(y)\}.$$
This surface also appeared  in \cite{st} and it is a doubly periodic surface by the periodicity of the functions $V(z/\sqrt{2})$ and $\sin(x)$.
\end{enumerate}

The second example corresponds with  the constants  
$$\begin{array}{lll}
 a_1=-1, & b_1=0,& c_1=1\\
 a_2=m, & b_2=-\frac{m+1}{2},& c_2=\frac{1-m}{2}\\
   a_3=-1,& b_3=1,& c_3=0,
   \end{array}$$
   
 where $m\in\r$. The integration of $f$ and $h$ is  
$$f(x)=\log(\sin(x)),\quad h(z)=-\log(\sin(z)). $$
For the function $g$, we have
$$\int^{e^g}\frac{d\tau}{\sqrt{-(m+1) \tau^4+2m\tau^2+1-m}}=\int^{e^g}\frac{d\tau}{\sqrt{(\tau^2-1)(-(m+1)\tau^2+1-m)}}=\frac{y}{\sqrt{2}}.$$
We discuss two cases:
\begin{enumerate}
\item Case $m=1$. Then $g(y)=-\log(\cosh(y))$ and   the surface is 
$$\{(x,y,z)\in\l^3: \sin(x)=\cosh(y)\sin(z)\}.$$
 Let us observe that the surface is doubly-periodic in the $xz$-plane, see Figure \ref{fig6}, middle. Moreover, 
\begin{eqnarray*}
X(u)+Y(v)-Z(w)&=&1+e^{-2u}-e^{2v}-e^{-2u-2v}=(e^{2v}-e^{-2u})(e^{-2v}-1)\\
&=&\frac{\sin(x)^2-\cosh(y)^2}{\sin(x)^2\cosh(y)^2}(\cosh(y)^2-1)\leq 0.
\end{eqnarray*}
Thus the surface is spacelike except in those points where $y=0$. By the equation of the surface, $\sin(x)=\sin(z)$ and this set of points are the straight-lines with equation
$\{y=0,z=x+2\pi\mathbb{Z}\}\cup \{y=0,z=\pi-x+2\pi\mathbb{Z}\}$. 
\item Case $m=0$. Now the elliptic integral coincides with (\ref{eqt4}) and  the equation of the surface is 
$$\{(x,y,z)\in\l^3: \sin(x)M(y/\sqrt{2})= \sin(z)\}.$$ 
This surface is doubly periodic because the functions $M(y/\sqrt{2})$ and $\sin(z)$ are periodic. In  Figure \ref{fig6}, right we show a piece of the surface. In fact, and up periodic translations,  the surface contains the pair of orthogonal spacelike lines $\{x=0,z=0\}\cup\{z=0,y=y_o\}$, with $M(y_o/\sqrt{2})=0$, $y_o\not=0$, where the surface can be repeated by symmetry reflections thanks to the reflection principle for maximal surfaces \cite{ac}. Also the surface  contains singularities at the points $(\pi/2,0,\pi/2)$ and $(-\pi/2,0,\pi/2)$ and their translations by periodicity.  
\end{enumerate}
 \begin{figure}[hbtp]
\begin{center}\includegraphics[width=.3\textwidth]{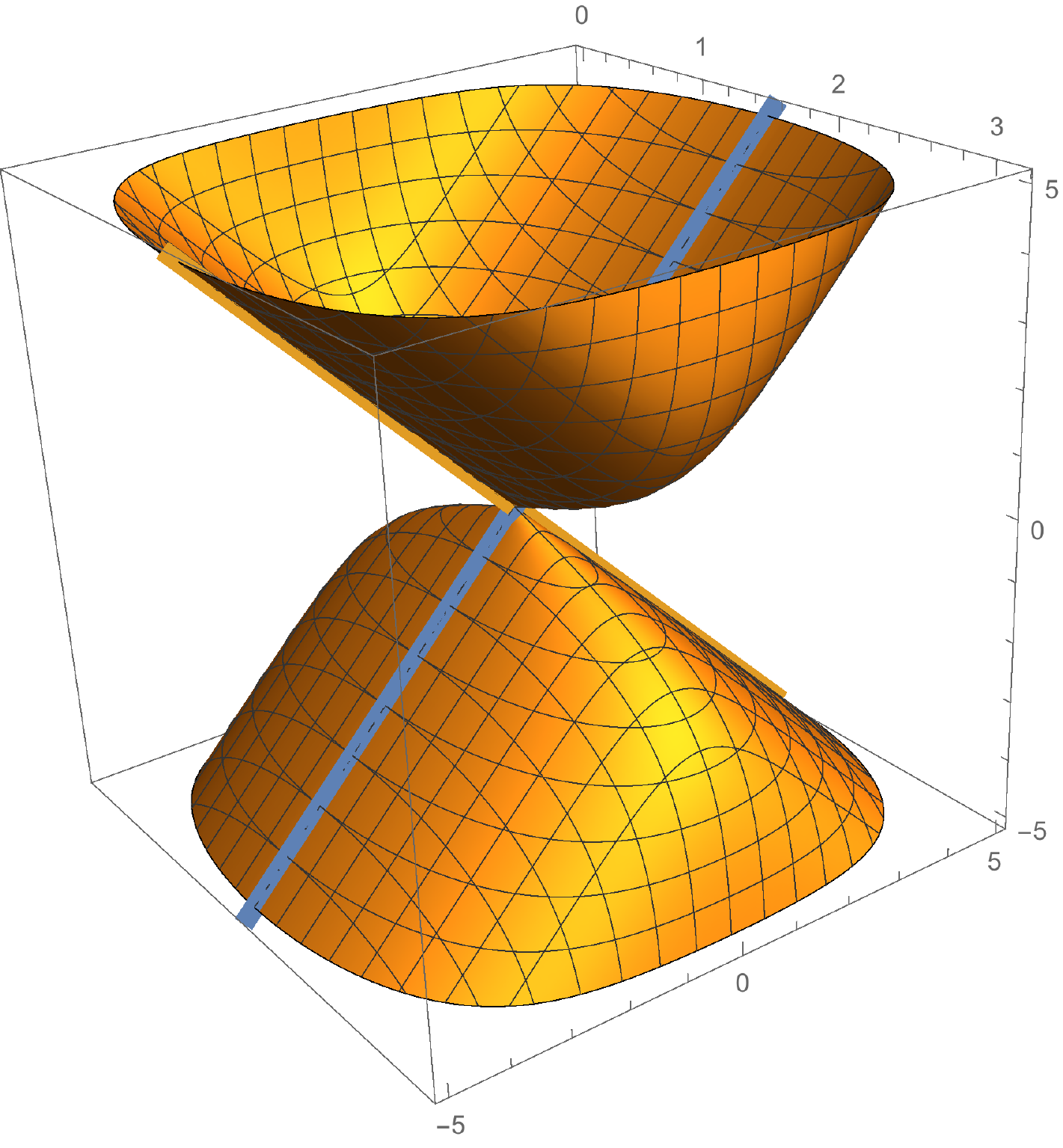}\, \includegraphics[width=.3\textwidth]{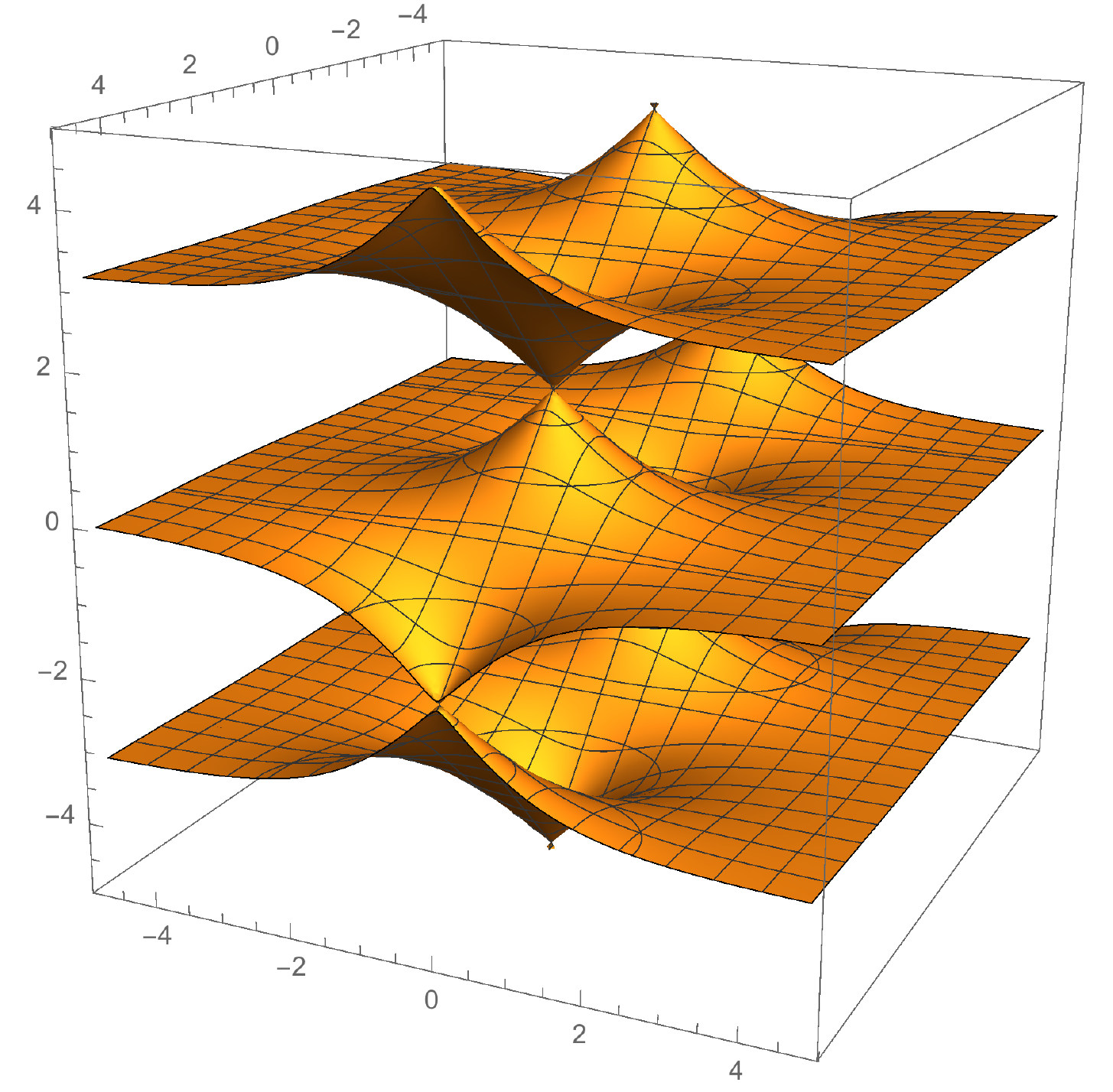}\, \includegraphics[width=.3\textwidth]{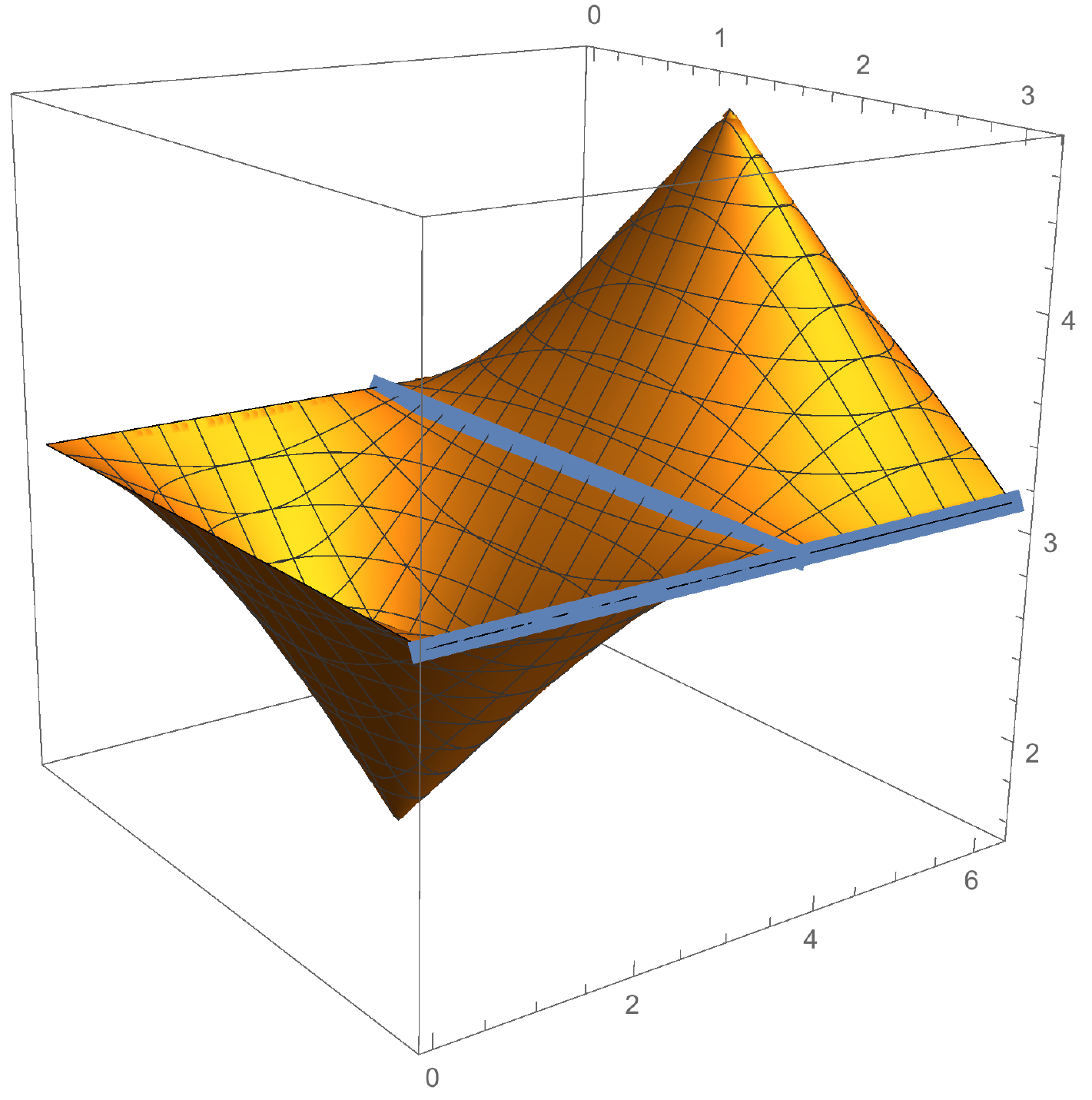}\end{center}
 \caption{Left: the surface  $\sin(x)\cosh(z)=\cosh(y)$. Middle: the surface  $\sin(x)=\cosh(y)\sin(z)$. Right. the surface  $\sin(x)M(y/\sqrt{2})= \sin(z)$ where we show the pair of straight-lines  through which the surface can be repeated by symmetries  }\label{fig6}
 \end{figure}

\subsection{Case that none of the three integrals can be solved by quadratures}

Consider  the constants in (\ref{var1}) given by 
$$\begin{array}{lll}
 a_1=-2m^2+1,& b_1=m^4,& c_1=1,\\
 a_2=-1,&  b_2=1,& c_2=m^2,\\
 a_3=1,&  b_3=1,&  c_3=m^2,
\end{array}$$
where $m\in\r$.   Then
$$X(u)=-2m^2+1+m^4 e^{2u}+e^{-2u},\quad Y(v)=-1+e^{2v}+m^2e^{-2v},\quad Z(w)=1+e^{2w}+m^2e^{-2w}.$$
The case $m=0$ appeared in the subsection \ref{s-421}.

Consider the interesting case   $m=1$. Now
$$X(u)=-1+e^{2u}+e^{-2u},\quad Y(v)=-1+e^{2v}+e^{-2v},\quad Z(w)=1+e^{2w}+e^{-2w}.$$
The solution of the first differential equation is  
$$\int\frac{df}{\sqrt{-1+e^{2f}+e^{-2f}}}=x.$$ Then
$$x=\int\frac{df}{\sqrt{-1+e^{2f}+e^{-2f}}}=\int^{e^f}\frac{d\tau}{\sqrt{1-\tau^2+\tau^4}}.$$
Denote $t=\mathcal{F}(\xi)$ the inverse of the elliptic integral $\int(1-\tau^2+\tau^4)^{-1/2}d\tau$. Similarly, let $\mathcal{G}(\xi)$   the inverse of the elliptic integral $ \int (1+\tau^2+\tau^4)^{-1/2}d\tau$.  Then the surface takes the form $\log\mathcal{F}(x)+\log\mathcal{F}(y)+\log\mathcal{G}(z)=0$, or equivalently
$$\mathcal{F}(x) \mathcal{F}(y) \mathcal{G}(z)=1.$$

\section{Separable ZMC  surfaces: case $K<0$.}\label{sec5}

In this section we obtain   examples of separable ZMC  surfaces when $K<0$ in Theorem \ref{t1}. Using Proposition   \ref{pr1} and without loss of generality, we  suppose   that $K=-1$ and $k=1$. We know that    the constants   $a_i$, $b_i$ and $c_i$   satisfy (\ref{var2}) and the functions $f$, $g$ and $h$ of \ref{s1}) are the solutions of the differential equations (\ref{sol3}).    In this section, we will only show some examples and pictures, see   Figure \ref{fig7}.   

\subsection{Example 1}
Consider the  constants
$$\begin{array}{lll}
 a_1=1, & b_1=0,& c_1=1\\
 a_2=1,& b_2=0,&c_2=1\\
 a_3=\frac12,& b_3=\frac12,& c_3=0.
   \end{array}$$
 Then 
 $$X(u)=1+\sin(u),\quad Y(v)=1+\sin(v),\quad Z(w)=\frac12+\frac12\cos(w).$$
The integrations of the three  equations yield
$$f(x)=2\arctan\left(\frac{1+\sinh(x/\sqrt{2})}{1-\sinh(x/\sqrt{2})}\right),$$ 
$$g(y)=2\arctan\left(\frac{1+\sinh(y/\sqrt{2})}{1-\sinh(y/\sqrt{2})}\right),$$
 $$h(z)= 2\arctan\left(\sinh(z/2)\right).$$
After some manipulations, the implicit equation of the surface  is 
$$
\frac{\sinh \left(x/\sqrt{2}\right) \sinh \left(y/\sqrt{2}\right)-1}{\sinh \left(x/\sqrt{2}\right)+\sinh \left(y/\sqrt{2}\right)}=-\sinh(z/2).$$

  \begin{figure}[hbtp]
\begin{center}\includegraphics[width=.3\textwidth]{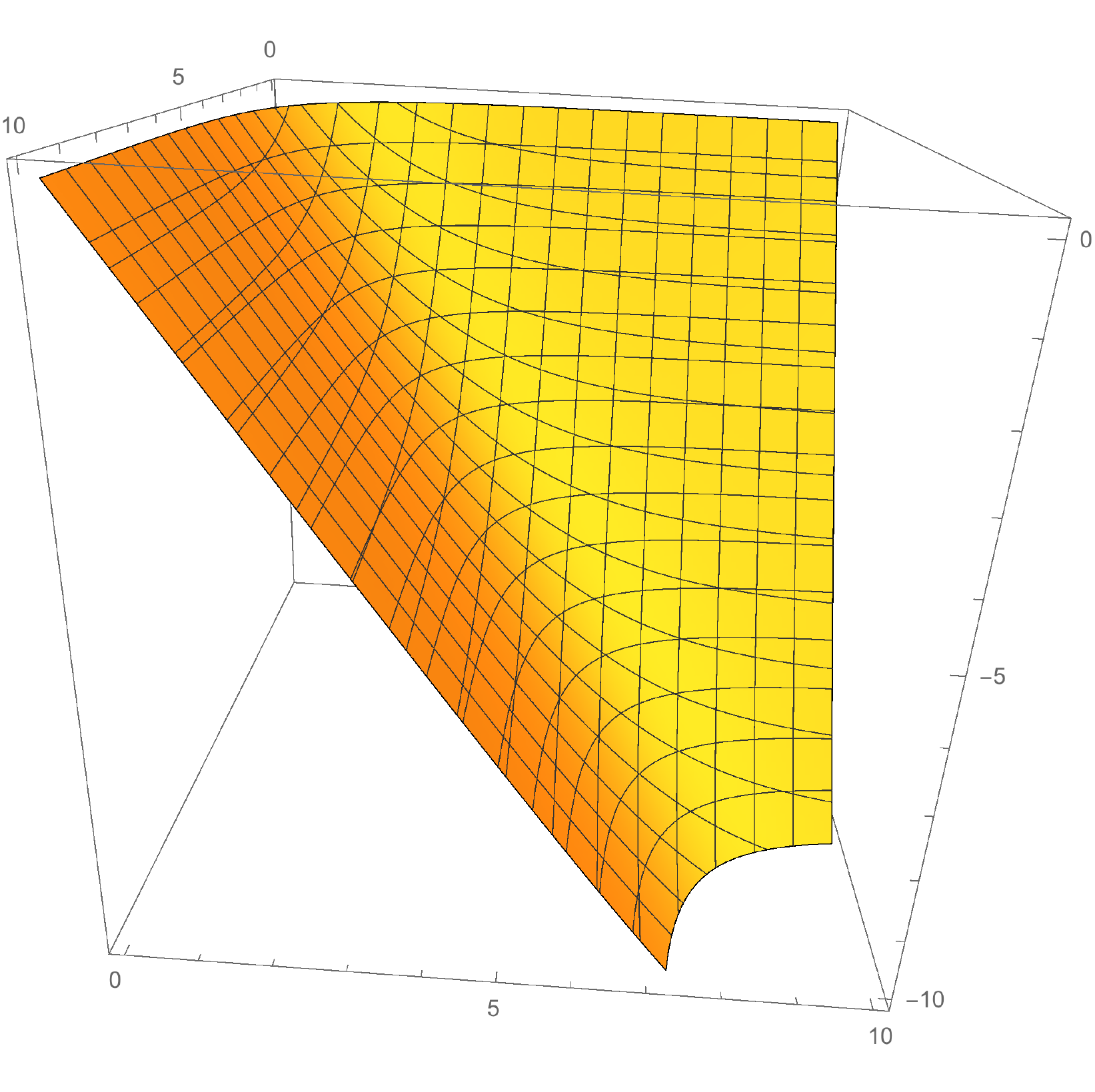}\quad \includegraphics[width=.3\textwidth]{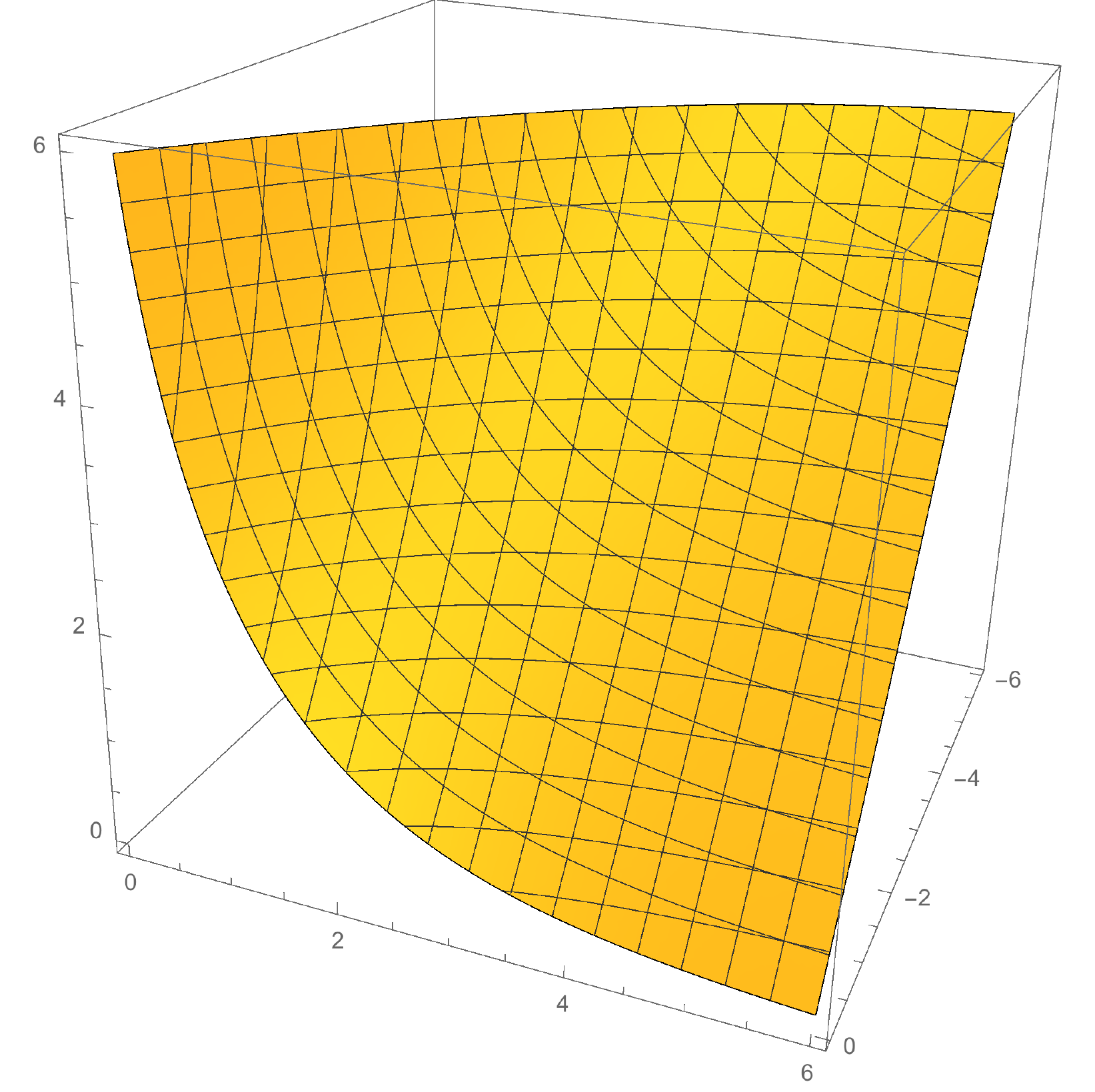}\end{center}
 \caption{The surfaces of Example 1 of Section \ref{sec5}}\label{fig7}
 \end{figure}
\subsection{Example 2}

 Consider the following constants: 
$$\begin{array}{lll}
 a_1=\frac12, \ b_1=\frac12,& c_1=0\\
 a_2=1, & b_2=0,& c_2=1\\
 a_3=\frac13, & b_3=0,& c_3=\frac13.
   \end{array}$$
  With a similar argument as the previous example, the integrations of the three differential equations by quadratures leads to
 $$f(x)=2\arctan\left(\sinh(x/2)\right),$$
 $$g(y)=2\arctan\left(\frac{1+\sinh(y/\sqrt{2})}{1-\sinh(y/\sqrt{2})}\right),$$
 $$h(z)= 2\arctan\left(\frac{1+\sinh(z/\sqrt{6})}{1-\sinh(z/\sqrt{6})}\right).$$
  Then the implicit equation of the surface is 
 $$
\frac{\sinh \left(y/\sqrt{2}\right) \sinh \left(z/\sqrt{6}\right)-1}{\sinh \left(y/\sqrt{2}\right)+\sinh \left(z/\sqrt{6}\right)}=-\sinh(x/2).$$


\end{document}